\documentclass[11 pt]{amsart}

\usepackage[cp1251]{inputenc}
\usepackage[english]{babel}
\usepackage{morefloats}
\usepackage{amssymb} 
\usepackage{amsfonts} 
\usepackage{amsmath}
\usepackage{mathrsfs}
\usepackage{amsthm} 
\usepackage{epsfig} 
\usepackage{color}
\usepackage{amscd}
\usepackage[all]{xy}
\usepackage{subfigure}

\newtheorem{lemma}{Lemma}[section]
\newtheorem{teo}[lemma]{Theorem}
\newtheorem{prop}[lemma]{Proposition}
\newtheorem{cor}[lemma]{Corollary} 
\newtheorem{conj}[lemma]{Conjecture}

\theoremstyle{definition}
\newtheorem{defn}[lemma]{Definition}

\newtheorem{example}[lemma]{Example}

\theoremstyle{remark}
\newtheorem{rem}[lemma]{Remark} 

\newcommand{\Iso}{{\rm Isom}}

\newcommand{\Conv}{{\rm Conv}}

\newcommand{\matR} {\ensuremath {\mathbb{R}}}

\newcommand{\matZ} {\ensuremath {\mathbb{Z}}}

\newcommand{\matH} {\ensuremath {\mathbb{H}}}

\newcommand{\block}{\ensuremath {\mathscr{B}}}
\newcommand{\cell}{\ensuremath {\mathscr{C}}}

\newcommand{\Vol}{{\rm Vol}}

\newcommand{\prf}[1]{\vspace{2pt}\noindent\textit{Proof of \ref{#1}}.\ }
\newcommand{\prfend}{{\hfill\hbox{$\square$}\vspace{2pt}}}

\author{Alexander Kolpakov}
\address{Department of Mathematics, 1326 Stevenson Center, Nashville, TN 37240, USA}
\email{kolpakov dot alexander at gmail dot com}

\author{Bruno Martelli}
\address{Dipartimento di Matematica ``Tonelli'', Largo Pontecorvo 5, 56127 Pisa, Italy}
\email{martelli at dm dot unipi dot it}

\thanks{The first named author is supported by the SNSF researcher scholarship PBFRP2-145885 and the SNSF project ``Discrete hyperbolic geometry'' no. 200020-144438/1. The second named author was supported by the Italian FIRB project ``Geometry and topology of low-dimensional manifolds'', RBFR10GHHH}

\title[]{Hyperbolic four-manifolds with one cusp}

\begin{document}

\begin{abstract}
We introduce an algorithm which transforms every four-dimensional cubulation into an orientable cusped finite-volume hyperbolic four-manifold. Combinatorially distinct cubulations give rise to topologically distinct manifolds.

Using this algorithm we construct the first examples of finite-volume hyperbolic four-manifolds with one cusp. More generally, we show that the number of $k$-cusped hyperbolic four-manifolds with volume $\leqslant V$ grows like $C^{V\ln V}$ for any fixed $k$. As a corollary, we deduce that the $3$-torus bounds geometrically a hyperbolic manifold.
\end{abstract}

\maketitle

\section*{Introduction}
By Margulis' Lemma, a finite-volume complete hyperbolic $n$-manifold $M^n$ has a finite number of ends called \emph{cusps}, each of which is diffeomorphic to $N^{n-1} \times [0, +\infty)$ for a certain closed connected flat $(n-1)$-manifold $N^{n-1}$. 

In dimension three, we may construct cusped hyperbolic manifolds in various ways, for instance by removing a knot or link complement from $S^3$. There are essentially two different techniques to prove that a link complement is hyperbolic: by decomposing it into geodesic ideal hyperbolic polyhedra, or by checking that the manifold does not contain an immersed essential surface with $\chi \geqslant 0$ and thus invoking geometrisation. The first method was used by Thurston in his notes \cite{Th}, where he constructed various hyperbolic $3$-manifolds with an arbitrary number of cusps. The computer program SnapPy \cite{SnapPy} may be used to check the hyperbolicity of any link with a reasonable number of crossings.

In higher dimensions, constructing hyperbolic manifolds is more complicated. Due to the absence of a geometrisation theorem of any kind, the hyperbolic structure on a smooth manifold needs to be established explicitly, and this is typically done either by arithmetic methods or by assembling geodesic polyhedra. The largest known census of cusped hyperbolic $4$-manifolds is the list produced by J.~Ratcliffe and S.~Tschantz \cite{RT} which contains $1171$ distinct manifolds, all obtained by pairing isometrically the faces of the ideal hyperbolic $24$-cell: these manifolds have either $5$ or $6$ cusps. 

We construct here the first example of a finite-volume hyperbolic four-manifold having only one cusp. One of the motivations for this work is a result by D.~Long and A.~Reid \cite{LR} which shows that, amongst the six diffeomorphism types of orientable flat $3$-manifolds, at least two of them cannot be cusp sections of a single-cusped four-manifold (but they are sections in some multi-cusped one \cite{N}). The authors then asked \cite{LR, LRall} whether any single-cusped hyperbolic manifold exists in dimension $n\geqslant 4$. The techniques introduced in the present paper answer this question in the affirmative if the dimension is $n=4$, but are not applicable in higher dimensions. We note that by a recent result of M.~Stover \cite{Stover} there are no single-cusped hyperbolic \textit{arithmetic} orbifolds in dimension $n\geqslant 30$. 

In the present paper, we show that there are plenty of single-cusped hyperbolic four-manifolds, and more generally of hyperbolic four-manifolds with any given number $k\geqslant 1$ of cusps. Let $\rho_k(V)$ be the number of pairwise non-homeomorphic orientable hyperbolic four-manifolds with $k$ cusps and volume at most $V$. The main result is the following.

\begin{teo} \label{V:intro:teo}
For every integer $k\geqslant 1$ there are two constants $C>1$, $V_0>0$ such that $\rho_k(V) > C^{V\ln V}$ for all $V>V_0$.
\end{teo}

Let $\rho(V)$ be the total number of hyperbolic four-manifolds with volume at most $V$: it was proved in  \cite{BGLM} that $C_1^{V\ln V} > \rho(V) > C_2^{V\ln V}$ for some constants $C_1 > C_2 > 1$. 

Following P.~Ontaneda \cite{O}, we say that a flat manifold \emph{bounds geometrically} a hyperbolic manifold if it is diffeomorphic to a cusp section of some single-cusped hyperbolic manifold. By analysing the cusp shapes we deduce the following corollary.
\begin{cor}
The $3$-torus bounds geometrically a hyperbolic manifold.
\end{cor}
In fact, Ontaneda has proved that every flat manifold bounds geometrically a negatively pinched Riemannian manifold \cite{O}, but Long and Reid showed that at least two among the six orientable flat $3$-manifolds cannot bound a hyperbolic manifold \cite{LR}. As we said above, the $3$-torus is the first example of a connected flat manifold of dimension $n\geqslant 3$ that bounds a hyperbolic manifold.

The proof of Theorem \ref{V:intro:teo} is constructive and proceeds as follows. The \emph{ideal hyperbolic $24$-cell} $\cell$ is a well-known ideal right-angled four-dimensional hyperbolic polytope with $24$ facets and $24$ ideal vertices: each facet is a regular ideal octahedron. The $24$ facets are naturally divided into three sets of $8$ facets each, which we colour correspondingly in green, red and blue. We produce four identical copies $\cell_{11}, \cell_{12}, \cell_{21}$ and $\cell_{22}$ of $\cell$ and identify the corresponding red and blue facets as described by the pattern below:
$$
\xymatrix{ 
\cell_{11} \ar@{<->}^R[r] \ar@{<->}_B[d] & \cell_{12} \ar@{<->}^B[d] \\
\cell_{21}  \ar@{<->}_R[r] & \cell_{22}}
$$
That is, we identify the red facets ``horizontally'' and the blue facets ``vertically''.
The resulting object is a four-dimensional complete hyperbolic manifold $\block$ with non-compact geodesic boundary. The geodesic boundary is formed by the $4\times 8 = 32$ green facets which were left un-paired and has eight components, each isometric to a well-known cusped hyperbolic three-manifold: the complement of the chain link shown in Fig.~\ref{chainlink:fig}.

\begin{figure}
 \begin{center}
  \includegraphics[width = 3.5 cm]{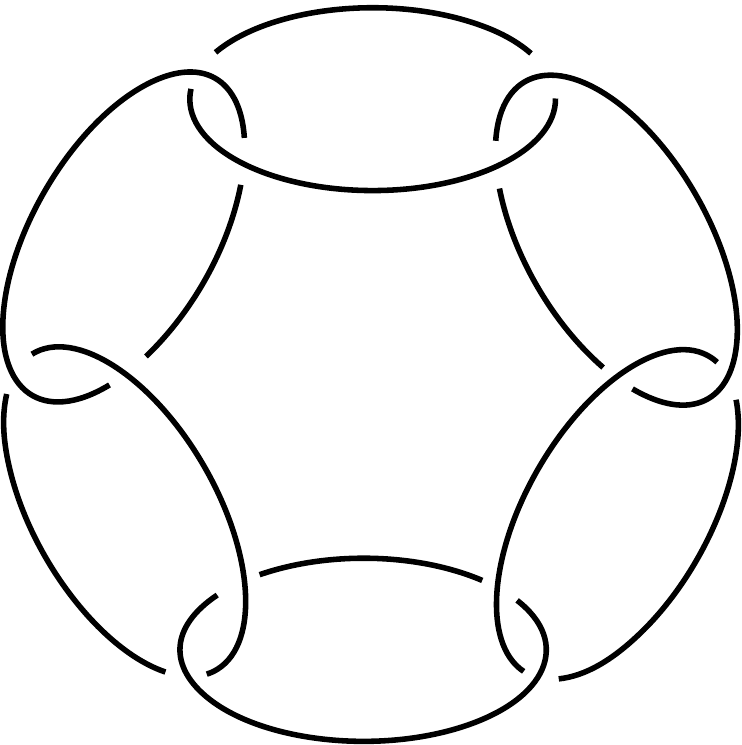}
 \end{center}
 \caption{The minimally twisted chain link with $6$ components. Its complement is tessellated by four regular ideal octahedra. The block $\block$ has eight geodesic boundary components, each isometric to this hyperbolic link complement.}
 \label{chainlink:fig}
\end{figure}

The block $\block$ has eight boundary components and $24$ cusps, each diffeomorphic to $S^1 \times S^1\times [0,1] \times [0,+\infty)$: this is a four-dimensional analogue of the annular cusps arising in dimension three. It turns out that $\block$ looks combinatorially much alike as a four-dimensional hypercube $H$: the eight boundary components correspond to the facets of $H$ and the $24$ cusps correspond to the $24$ two-dimensional faces of $H$. This combinatorial correspondence preserves all the geometric adjacencies. 

As usual, we define a four-dimensional \emph{cubulation} as the combinatorial data that consists of $n$ (four-dimensional) hypercubes and an isometric pairing of the resulting $8n$ facets. Having noticed that $\block$ looks like a hypercube, we may transform every cubulation into an orientable complete finite-volume cusped hyperbolic four-manifold by substituting every hypercube with an instance of $\block$ and glueing the geodesic boundaries as described by the combinatorics of the facet pairing. This construction was inspired by a similar algorithm introduced in \cite{CFMP}, which transforms a triangulation into a hyperbolic $3$-manifold. We have constructed a map
$$\big\{ {\rm cubulations} \big\} \longrightarrow \left\{ \begin{matrix} {\rm orientable\ complete\ finite-volume} \\ {\rm cusped\ hyperbolic\ four-manifolds} \end{matrix} \right\}.$$

Various informations on the topology and geometry of the resulting hyperbolic four-manifold $M$ can be derived directly from the cubulation. The volume of $M$ is $\frac {16 n}3 \pi^2$ where $n$ is the number of hypercubes in the cubulation. The cusps of $M$ may be recovered as follows: the facet-pairings in the cubulation induce a partition of the $24n$ two-dimensional faces into cycles, and every cycle corresponds to a cusp. Therefore, in order to construct a hyperbolic four-manifold with one cusp we only need to construct a cubulation where all two-dimensional faces become identified (shortly: a cubulation with only one $2$-face). This can be done with a single hypercube. 

The Mostow-Prasad rigidity together with the Epstein-Penner canonical decomposition \cite{E} ensure us that combinatorially distinct cubulations produce non-homeomorphic four-manifolds. In other words, the map from the set of cubulation into the set of cusped hyperbolic manifolds described above is injective. Therefore, in order to prove that $\rho_1 (V) \geqslant C^{V\ln V}$ we only need to show that the number of cubulations with $n$ hypercubes and one $2$-face grows faster than $C^{n\ln n}$.

The algorithm that transforms cubulations into hyperbolic manifolds can also be used to construct plenty of \emph{closed} Riemannian four-manifolds with non-positive sectional curvature or with Einstein metrics. It suffices to construct a cubulation where every cusp is homeomorphic to a $3$-torus (this condition is easily checked combinatorially) and then perform a Dehn filling, \emph{i.e.}~glue a copy of $D^2\times T^2$ at each cusp. The Dehn filling is encoded by a triple $(p,q,r)$ of co-prime integers, and if the triple is sufficiently complicated the resulting manifold admits a non-positively curved metric (by Gromov-Thurston's $2\pi$ theorem, see \cite{A}) and even an Einstein metric thanks to a theorem of Anderson \cite{A} that extends Thurston's Dehn filling theorem to all dimensions. 

Thus, various closed Einstein four-manifolds can be constructed on the basis of a simple combinatorial data: this can be seen as an analogue of presenting closed hyperbolic $3$-manifolds as Dehn surgeries along links in $S^3$. 

\subsection*{Structure of the paper}
We introduce the building block $\block$ in Section \ref{block:section}, and then use it in order to transform cubulations into hyperbolic four-manifolds in Section \ref{cubulations:section}. We apply this construction in Section \ref{manifolds:section} completing the proof of Theorem \ref{V:intro:teo}. Finally, we discuss Dehn fillings in Section \ref{Dehn:section}.

\section{The building block} \label{block:section}
We define the object which will play the central r\^{o}le in the sequel, namely, the building block $\block$. This is a hyperbolic four-dimensional finite-volume manifold with non-compact totally geodesic boundary. In the following section we shall use $\block$ in order to transform any cubulation into a hyperbolic four-manifold.

\subsection{The octahedral $3$-manifold} \label{O:subsection}
Let us start with the description of a cusped hyperbolic $3$-manifold whose eight disjoint isometric copies will form the boundary of the block $\block$.

Let $O$ be a regular ideal hyperbolic octahedron. Let us colour the faces of $O$ in blue and red in the chequerboard fashion (thus, every edge of $O$ is adjacent to a red and a blue triangle). Now we take four identical copies $O_{11}, O_{21}, O_{12}$ and $O_{22}$ of $O$ and pair their faces following the rules below:

\begin{itemize}
\item for $i\in \{1,2\}$ we glue each red face of $O_{i1}$ to the corresponding red face of $\cell_{i2}$;
\item for $j\in \{1,2\}$ we glue each blue face of $O_{1j}$ to the corresponding blue face of $\cell_{2j}$.
\end{itemize}
The facets are matched by identifying all the pairs of corresponding points in them by means of a hyperbolic isometry.
The rules are summarised in the following glueing diagram:
\begin{equation} \label{O:eqn}
\xymatrix{ 
O_{11} \ar@{<->}^R[r] \ar@{<->}_B[d] & O_{12} \ar@{<->}^B[d] \\
O_{21}  \ar@{<->}_R[r] & O_{22}}
\end{equation}

\begin{figure}
 \begin{center}
  \includegraphics[width = 11.5 cm]{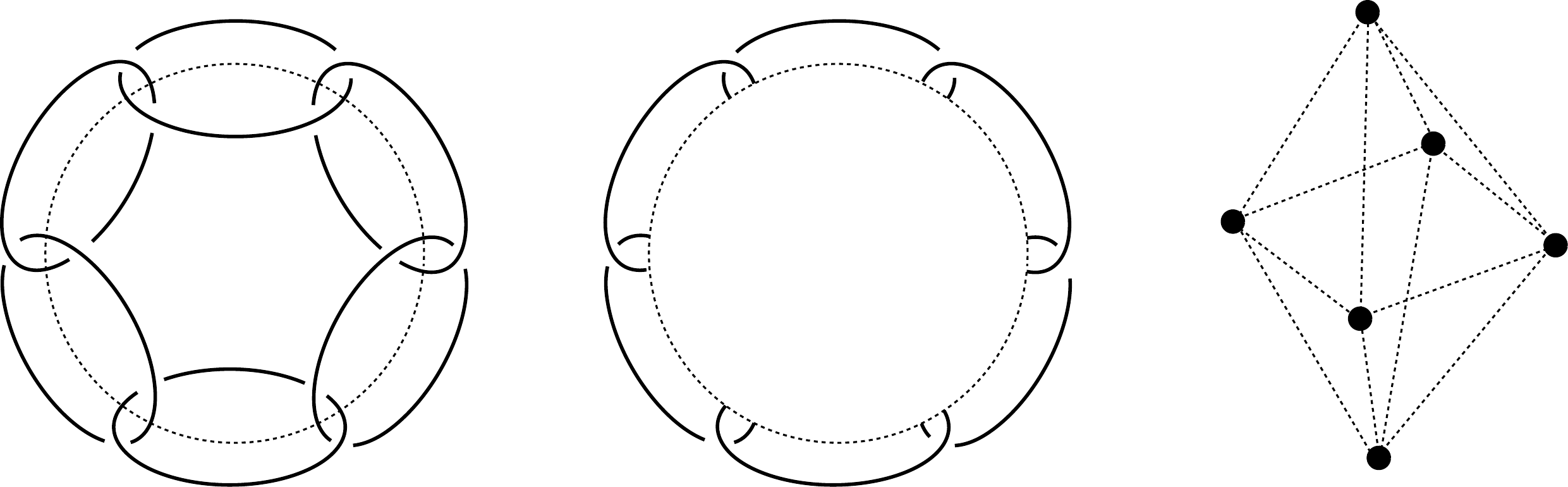}
 \end{center}
 \caption{The minimally twisted chain link with 6 components (left) is symmetric with respect to a $\pi$-rotation $\iota$ about the dotted circle. This symmetry quotients the hyperbolic link complement $N$ down to the octahedral orbifold (right), obtained from the central picture by contracting each solid arc to a vertex. All the edges in the orbifold on the right have index $2$.}
 \label{hexatangle:fig}
\end{figure}

This glueing clearly gives rise to a hyperbolic $3$-manifold $N$, since the dihedral angle along each edge in $O$ equals $\frac \pi 2$ and the edges are assembled into sets of four elements each. We call $N$ the \emph{octahedral manifold}. We may define an involution $\iota$ by interchanging $O_{11}$ with $O_{22}$ and $O_{12}$ with $O_{21}$. The quotient orbifold $N/_\iota$ may be then described as
$$
\xymatrix{ 
O_{11} \ar@<1ex>@{<->}[r]^R \ar@<-1ex>@{<->}_B[r] & O_{12}. 
}
$$
It is tessellated by two isometric octahedra, with all the corresponding faces identified. Therefore $N/_\iota$ is the octahedral orbifold shown in Fig.~\ref{hexatangle:fig}-(right) with base space $S^3$ and singular locus the $1$-skeleton of an octahedron (all of its edges are labelled with the index $2$). The manifold $N$ is a double cover of that orbifold, and Fig.~\ref{hexatangle:fig} shows that $N$ is homeomorphic to the complement of the minimally twisted six chain link. 

\begin{figure}
 \begin{center}
  \includegraphics[width = 3 cm]{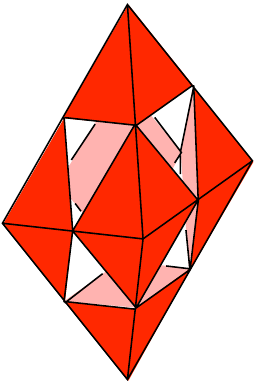}
 \end{center}
 \caption{The regular ideal octahedron has a maximal horocusp section that consists of six Euclidean unit squares.}
 \label{cusp_section2:fig}
\end{figure}

The regular ideal octahedron $O$ has a maximal horocusp section that consists of six Euclidean unit squares, see Fig.~\ref{cusp_section2:fig}. The maximal horocusp sections of $O_{11}, O_{12}, O_{21}$ and $O_{22}$ glue together up to a maximal horocusp section of $N$ that consists of six flat tori, each isometric to the square torus $T$ of area $4$, constructed by identifying the opposite edges of a $2\times 2$ square. The six cusps of $N$ are in a natural $1-1$ correspondence with the ideal vertices of the octahedron $O$. 

Let $H$ (resp.~$V$) be the isometry of $N$ that interchanges $O_{i1}$ with $O_{i2}$ (resp.~$O_{1j}$ with $O_{2j}$) for all $i$ (resp.~$j$). The isometries $H$ and $V$ are orientation-reversing and we have $\iota = H\cdot V = V\cdot H$.

\begin{prop} \label{N:prop}
Every isometry of $N$ induces an isometry of $O$. Moreover, there is the following exact sequence
$$0 \longrightarrow \matZ_2 \oplus \matZ_2 \longrightarrow \Iso(N) \longrightarrow \Iso(O) \longrightarrow 0,$$ 
where $\matZ_2 \oplus \matZ_2$ is generated by $H$ and $V$.
\end{prop}
\begin{proof}
The four octahedra form the canonical Epstein-Penner decomposition \cite[Theorem~3.6]{E} of $N$ and therefore are preserved by any isometry. On the other hand, every isometry $\varphi$ of $O$ is realised by an isometry of $N$: if $\varphi$ preserves the red-blue colouring, then it suffices to act by $\varphi$ on each $O_{ij}$, if it inverts the colouring then we act by $\varphi$ on each $O_{ij}$ and then exchange $O_{11}$ with $O_{22}$, in order to invert the colourings in the square diagram (\ref{O:eqn}), so that the resulting isometry is well-defined on $N$.

The kernel of the surjective map $\Iso(N) \to \Iso(O)$ consists of all isometries of $N$ that fix the cusps. These are naturally identified with the symmetries of the square (\ref{O:eqn}) that preserve the red-blue colouring. Indeed, the kernel is the $\matZ_2\oplus \matZ_2$ group generated by $H$ and $V$.
\end{proof}

\begin{cor} The involution $\iota$ is the unique orientation-preserving non-trivial isometry of $N$ that fixes the cusps.
\end{cor}

The action of $\iota$ on each cusp is non-trivial: the cusp shape is a square torus, and $\iota$ acts like a reflection with respect to the centre of the square. Therefore $\iota$ acts as an elliptic involution, whose effect on the homology is multiplication by $-1$.

Let $T\subset M$ be a torus inside an orientable closed three-manifold. The operation of cutting $M$ along $T$ and re-glueing back using an elliptic involution is sometimes called a \emph{mutation}: the result of this operation is a new orientable three-manifold, which often is not homeomorphic to $M$. Here we introduce a similar operation for hyperbolic four-manifolds.

\begin{defn} \label{mutation:defn}
Let $M$ be an orientable hyperbolic four-manifold which contains a three-dimensional geodesic sub-manifold $N$ isometric to the octahedral manifold. Let us call a \emph{mutation} of $M$ along $N$ the operation of cutting $M$ along $N$ and re-glueing it back via the involutary isometry $\iota$. The result of this operation is a hyperbolic four-manifold, which is typically non-homeomorphic to the initial one.
\end{defn}

\begin{rem}
An embedded cusp section $X$ of $M$ is a collection of three-dimensional flat manifolds that intersect the geodesic three-manifold $N$ along six flat tori. A mutation of $M$ along $N$ changes the cusp section $X$ via mutation along these tori, because $\iota$ acts on the cusps like an elliptic involution.
\end{rem}

\subsection{The $24$-cell} \label{euclidean:cell:subsection}
The \emph{24-cell} $\cell$ is the only regular polytope in all dimensions $n\geqslant 3$ which is self-dual and not a simplex. It may be defined as the convex hull 
$$\cell = \Conv (V)$$
of the set $V$ that contains $24$ points in $\matR^4$ obtained by permuting the coordinates of
$$\left(\pm 1, \pm 1, 0, 0 \right).$$
These $24$ points are the vertices of $\cell$. One checks easily that $\cell$ has $24$ facets, whose affine supporting hyperplanes are
$$\{\pm x_i = 1\}, \qquad \left\{\pm \frac {x_1}2 \pm \frac{x_2}2 \pm \frac{x_3}2 \pm \frac{x_4}2 = 1 \right\}.$$
Each facet is a regular octahedron. The dual polytope $\cell^*$ is therefore the convex hull 
$$\cell^* = \Conv(G\cup R \cup B)$$
where $G$ contains the $8$ points obtained by permuting the coordinates of
$$(\pm 1, 0, 0, 0)$$
and $R\cup B$ contains $16$ points of the form
$$\left(\pm \frac 12, \pm \frac 12, \pm \frac 12, \pm \frac 12\right).$$
It is convenient to partition the latter set into $R\sqcup B$ where $R$ (resp.~$B$) is the subset of $8$ points having an even (resp.~odd) number of minus signs. 

The facets of $\cell$ are regular octahedra in a $1-1$ correspondence with the vertices $G\cup R \cup B$ of $\cell^*$ and we colour them accordingly in green, red and blue. This three-colouring of $\cell$ is indeed natural: one can prove that every symmetry of $\cell$ induces a permutation of the sets $G, R$ and $B$, and vice versa every permutation may be realised in this way. 

Another fact worth mentioning is that  $\Conv (R\cup B)$ represents a hypercube and $\Conv(G)$ corresponds to its dual $16$-cell, also the same is true after permuting the sets $R$, $B$ and $G$.

The $24$-cell $\cell$ is self-dual, \emph{i.e.}~it has a homothetic dual $\cell^*$. Thus $\cell$ has $24$ facets, $96$ triangular two-dimensional faces, $96$ edges and $24$ vertices. A facet is a regular octahedron, and (in accordance with the self-duality) each vertex figure is a cube.

\subsection{The hypercube}
Let us consider the hypercube
$$H = [-1,1]^4.$$
We have already noticed that the barycentre of a facet in $H$ is a point in $G$. Thus we get a natural $1-1$ correspondence
$$\big\{{\rm facets\ of \ } H\big\} \longleftrightarrow G \longleftrightarrow \big\{ {\rm green\ facets\ of\ } \cell \big\}.$$
The vertices in $V$ are precisely the barycentres of the $2$-dimensional faces in $H$, and thus we get one more $1-1$ correspondence
$$\big\{{\rm 2-faces\ of \ } H\big\} \longleftrightarrow V \longleftrightarrow \big\{ {\rm vertices\ of\ } \cell \big\}.$$
 
Said in a single phrase, the $24$-cell with its green facets and its vertices looks like a hypercube with its (cubic) facets and its (square) $2$-faces. This analogy is the core of our construction.
 
Restricted to the facets, this analogy is just the duality of polyhedra: an octahedral green facet of $\cell$ is dual to a cubic facet of $H$, both contained in the same affine hyperplane $x_i = \pm 1$. The duality map sends the vertices of the octahedron to the square faces of the cube.
 
\begin{rem} Although we will not use it here, we mention that the analogy extends to all the strata of $H$, as follows:
$$\big\{{\rm vertices\ of \ } H\big\} \longleftrightarrow B\cup R \longleftrightarrow \big\{ {\rm blue\ and\ red\ facets\ of\ } \cell \big\}.$$
$$\big\{{\rm edges\ of \ } H\big\} \longleftrightarrow 
\left\{ \begin{matrix}{\rm triangular\ 2-faces\ of\ } \cell \\ {\rm separating\ blue\ and\ red\ facets} \end{matrix}\right\}.$$
\end{rem}

\subsection{The regular ideal hyperbolic $24$-cell}
Every $n$-dimensional regular polytope $P$ has a hyperbolic ideal presentation obtained by normalising the coordinates of its vertices so that they lie on the unit sphere $S^{n-1}$ and by interpreting $S^{n-1}$ as the ideal boundary of $\matH^n$ in Klein's ball model. 

Note that the vertex figure of an ideal vertex is a Euclidean regular $(n-1)$-dimensional polytope in some horosphere, whose dihedral angles coincide with the dihedral angles of $P$. For that reason, the ideal regular hyperbolic tetrahedron, cube, octahedron and dodecahedron have Coxeter dihedral angles $\frac{\pi}3$, $\frac{\pi}3$, $\frac{\pi}2$ and $\frac \pi 3$, respectively, since their vertex figures are either equilateral triangles or squares. 

In what follows we keep denoting by $\cell$ the ideal regular hyperbolic $24$-cell obtained from the Euclidean 24-cell $\cell$ defined in Section \ref{euclidean:cell:subsection} in the way described above. The vertex figure of $\cell$ is a Euclidean cube and therefore $\cell$ has all dihedral angles $\frac \pi 2$. The $24$-cell is the unique regular four-dimensional polytope having cubical vertex figures: the vertex figures of the other five regular four-dimensional polytopes are other Platonic solids, and therefore their dihedral angles are not sub-multiples of $\pi$. Hence $\cell$ may be used as a building block in order to construct cusped hyperbolic $4$-manifolds, as shown by J.~Ratcliffe and S.~Tschantz \cite{RT}.

The boundary of $\cell$ consists of $24$ regular ideal hyperbolic octahedra, $96$ ideal triangular $2$-dimensional faces and $96$ geodesic edges. 

Recall that the octahedral facets of $\cell$ are coloured in green, red and blue. We now glue four isometric copies of $\cell$ together to produce a hyperbolic $4$-manifold $\block$ with totally geodesic boundary. 

\subsection{The $24$-cell block $\block$} 
In Section \ref{O:subsection} we constructed the octahedral hyperbolic $3$-manifold $N$ by glueing four copies of the regular ideal octahedron $O$ according to the diagram (\ref{O:eqn}), and by employing the bi-colouring of $O$. Now we construct the building block $\block$ from $\cell$ exactly in the same way, using the colouring of its facets.

We pick four isometric copies of $\cell$, which we denote by $\cell_{11}, \cell_{12}, \cell_{21}$ and $\cell_{22}$, then pair some of their facets as follows:
\begin{itemize}
\item for $i\in \{1,2\}$ we glue each red facet of $\cell_{i1}$ to the corresponding red facet of $\cell_{i2}$;
\item for $j\in \{1,2\}$ we glue each blue facet of $\cell_{1j}$ to the corresponding blue facet of $\cell_{2j}$.
\end{itemize}
The facets are matched by identifying all the pairs of corresponding points in them by means of a hyperbolic isometry.
We have glued together the red and the blue facets according to the same square diagram as (\ref{O:eqn}):
$$
\xymatrix{ 
\cell_{11} \ar@{<->}^R[r] \ar@{<->}_B[d] & \cell_{12} \ar@{<->}^B[d] \\
\cell_{21}  \ar@{<->}_R[r] & \cell_{22}}
$$

We denote by $\block$ the resulting topological object.
The identifications have paired all the blue and red facets: the un-paired facets of $\block$ are therefore the $4\times 8 = 32$ remaining green facets. As above, the fact that $\cell$ is right-angled should guarantee that the resulting object is a hyperbolic manifold: now we prove this in detail.

\begin{prop} \label{manifold:prop}
The space $\block$ is a hyperbolic four-manifold with totally geodesic boundary. 
\end{prop}
\begin{proof}
The building block $\block$ is obtained from $\cell_{11}, \cell_{21}, \cell_{12}$ and $\cell_{22}$ by an isometric glueing of some pairs of their facets. Let us consider $\block$ as a cell complex. Then the $4$- and $3$-dimensional strata of $\block$ clearly have a hyperbolic structure (with geodesic boundary). We need to check that the structure extends to any point $x$ lying in the $1$- or $2$-dimensional stratum $S$, \emph{i.e.}~on an edge or in an ideal triangle.

We may suppose that $x$ lies in a boundary edge or a triangle $S$ of $\cell_{11}$. Let us represent $\cell_{11}$ in the upper half-space model for $\matH^4$ and send one of the ideal vertices of $S$ to the infinity. Let $U_x$ be a horizontal horosphere passing through $x$. Suppose for now that the fourth coordinate $x_4$ of $x$ is big enough, so that the intersection $U_x\cap \cell_{11}$ is a Euclidean cube $C_{11}$. The point $x$ is either a vertex of or contained in an edge of $C_{11}$.

The faces of the cube $C_{11}$ are coloured in green, red and blue according to the colours of the facets of $\cell_{11}$ they are contained in. Opposite faces share the same colour. The block $\block$ contains four isometric copies $C_{11}, C_{12}, C_{21}, C_{22}$ of this cube, which are glued according to the same pattern as above:
$$
\xymatrix{ 
C_{11} \ar@{<->}^R[r] \ar@{<->}_B[d] & C_{12} \ar@{<->}^B[d] \\
C_{21}  \ar@{<->}_R[r] & C_{22}}
$$
These cubes form a flat manifold in $\block$, which is isometric to $T \times [0,1]$, where $T$ is the Euclidean square torus of area $4$, obtained by identifying the opposite sides of a $2\times 2$ square. Therefore $x$ is contained in a hyperbolic cusp based on that flat manifold: the hyperbolic structure clearly extends to $x$. If $x_4$ is arbitrary, then the intersection $\block \cap U_x$ is only a subset of a cube, but still it looks geometrically like a piece of a cube near $x$ and hence the same argument applies.
\end{proof}

During the proof we have described the cusps of $\block$: they are of the type $T\times [0,1] \times [0,+\infty)$, and they are $24$ in total, one for each ideal vertex of $\cell$. We will return to that later: first we describe the totally geodesic boundary of $\block$.

\begin{prop}
The block $\block$ has eight totally geodesic boundary components, each isometric to the octahedral $3$-manifold.
\end{prop}
\begin{proof}
A green octahedral facet $O$ of $\cell$ gives rise to four regular ideal octahedra $O_{11}, O_{12}, O_{21}$ and $O_{22}$ glued together in accordance with the square diagram (\ref{O:eqn}), forming an octahedral $3$-manifold. The eight green facets of $\cell$ produce eight such manifolds.
\end{proof}

The eight boundary components naturally correspond  to the green facets of $\cell$. There is a sequence of $1-1$ correspondences (recall that $H$ is the hypercube):
$$\big\{{\rm facets\ of \ } H\big\} \longleftrightarrow G \longleftrightarrow 
\big\{{\rm green\ facets\ of\ } \cell \big\} \longleftrightarrow 
\left\{\!\!\! \begin{array}{c} 
{\rm (geodesic)\ boundary} \\ {\rm components\ of\ } \block \end{array} \!\!\! \right\}$$

\subsection{The maximal cusp section}
Finally, we describe the cusps of $\block$. As an ideal regular polytope, the $24$-cell $\cell$ has a maximal horosection which meets every boundary octahedron also in a maximal (two-dimensional) horosection. The maximal horosection of an ideal regular octahedron clearly consists of six unit squares. Hence the maximal horosection of $\cell$ comprises a unit Euclidean cube $C$ for each vertex $v$ of $\cell$. 

The faces of $C$ inherit the colour of the facets of $\cell$ that they are contained in: every cube is hence coloured in green, red or blue, with opposite faces sharing the same colour. The block $\block$ contains four isometric copies $C_{11}, C_{12}, C_{21}$ and $C_{22}$ of  $C$, which are glued in accordance to the usual pattern:
$$
\xymatrix{ 
C_{11} \ar@{<->}^R[r] \ar@{<->}_B[d] & C_{12} \ar@{<->}^B[d] \\
C_{21}  \ar@{<->}_R[r] & C_{22}}
$$
These four cubes glue together in order to form the flat manifold $T \times [0,1]$ where $T$ is the Euclidean torus obtained by identifying the opposite sides of a $2\times 2$ square. The maximal horosection of $\cell$ then gives rise to a maximal horosection of $\block$ made of $24$ components (one for each vertex of $\cell$), each isometric to $T\times [0,1]$.

A horosection isometric to $T\times [0,1]$ bounds a \emph{toric cusp} homeomorphic to $T\times [0,1]\times [0,+\infty)$: this is a four-dimensional analogue to the annular cusps that one may find in hyperbolic $3$-manifolds with totally geodesic boundary, see fig.~\ref{toric_cusp:fig}. 

\begin{figure}
 \begin{center}
  \includegraphics[width = 12 cm]{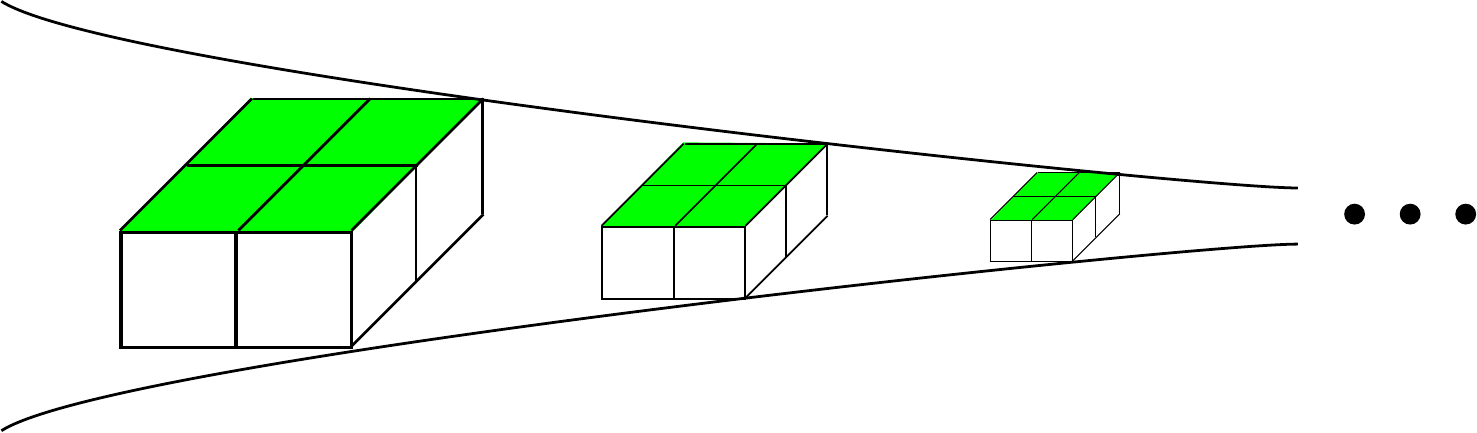}
 \end{center}
 \caption{A toric cusp in $\block$. Every flat section is $T\times [0,1]$, up to a similarity, where $T$ is the square torus obtained by identifying the opposite sides of a $2\times 2$ square. The flat section is tessellated by four cubes: these are four copies of the same vertex cube in $\cell$, glued following the pattern as in (\ref{O:eqn}).}
 \label{toric_cusp:fig}
\end{figure}

Thus, concerning the geodesic boundary components, we have found a sequence of natural $1-1$ correspondences:
$$\big\{{\rm 2-faces\ of \ } H\big\} \longleftrightarrow V \longleftrightarrow \big\{ {\rm vertices\ of\ } \cell \big\}  \longleftrightarrow \big\{ {\rm (toric)\ cusps\ of\ } \block \big\} .$$

Summing up, the block $\block$ looks combinatorially like the hypercube $H$, with $8$ geodesic boundary components corresponding to the facets of $H$, and $24$ toric cusps corresponding to the $2$-faces of $H$. (The vertices and edges of $H$ do not play a r\^{o}le here.) The correspondence is easily described on the facets: a geodesic boundary component $N$ is made of four copies of an octahedral green facet $O$ of $\cell$, which is dual to the corresponding cubic facet $C$ of $H$. The six cusps of $N$ correspond to the six vertices of $O$ and hence, by duality, to the square faces of $C$.

The picture is similar to the one given in \cite{CFMP} where the authors constructed a $3$-dimensional block combinatorially equivalent to a tetrahedron. In that paper this correspondence was used to transform any $3$-dimensional triangulation into an orientable hyperbolic cusped $3$-manifold. Here we perform an analogous construction, using hypercubes instead of tetrahedra.

\section{Cubulations} \label{cubulations:section}
We now construct orientable hyperbolic four-manifolds by glueing several copies of the block $\block$ along their totally geodesic boundaries. The combinatorial tool which is best suited to describe this procedure is a cubulation.

\subsection{The construction} \label{construction:subsection}
A \emph{combinatorial four-dimensional cubulation} (here for short, a \emph{cubulation}) is a data that consists of $n$ copies $H_1, \ldots, H_n$ of the standard hypercube $H$ together with an isometric pairing of the given $8n$ facets. An isometric pairing is a partition of the $8n$ facets into pairs, together with a Euclidean isometry between the two cubes in each pair. 

Here we show how a cubulation determines a finite-volume cusped orientable hyperbolic four-manifold, unique up to some well-understood mutations (recall Definition \ref{mutation:defn}). Note that we make no requirements on the cubulation: the topological space obtained by glueing the hypercubes $H_1,\ldots, H_n$ does not need to be a manifold. However, we always assume that it is connected.

First, we pick $n$ isometric copies $\block_1, \ldots, \block_n$ of the block $\block$.
Recall that there is a $1-1$ correspondence between the facets (resp., $2$-faces) of $H$ and the boundary components (resp., cusps) of $\block$. For every isometric pairing $\varphi\colon C^1 \to C^2$ of two cubic facets we construct an isometry $\varphi_*\colon N^1 \to N^2$ between the corresponding geodesic boundary components. Recall that $N^h$ is made of four copies $O_{ij}^h$ of an ideal octahedron $O^h$ naturally dual to the cube $C^h$, for every $h=1,2$. The isometry $\varphi$ defines an isometry $\varphi\colon O^1 \to O^2$ according to the duality.

By Proposition \ref{N:prop}, there are precisely four distinct choices for $\varphi_*$.
Let $\varphi_*^{ij}\colon N^1 \to N^2$ be the isometry that sends $O^1_{11}$ to $O^2_{ij}$ via $\varphi$ and extends (uniquely) to the whole of $N^1$. The four choices are 
$$\varphi_*^{11}, \quad \varphi_*^{12} = H\circ \varphi_*^{11}, \quad \varphi_*^{21} = V\circ \varphi_*^{11},  \quad \varphi_*^{22} = \iota \circ \varphi_*^{11}.$$
Amongst these isometries, two are orientation-preserving and the other two are orientation-reversing; the isometries $\varphi_*^{11}$ and $\varphi_*^{22}$ are of the same type and differ only by the involution $\iota$, as are the isometries $\varphi_*^{12}$ and $\varphi_*^{21}$.
The following lemma follows from the construction.

\begin{lemma} \label{orientation:lemma}
The map $\varphi_*^{11}$ is orientation-reversing if and only if $\varphi$ is.
\end{lemma}

Since we want to obtain an orientable four-manifold, we choose one of the two orientation-reversing isometries.

The possible two choices differ exactly by composition with the involution $\iota$. That is, the resulting orientable four-manifold $M$ is uniquely determined, up to mutations along some of the $4n$ geodesic octahedral manifolds which it contains by construction.

Summing up, we have described an algorithm that transforms every cubulation into a cusped hyperbolic orientable four-manifold, well-defined up to a mutation. We now study these hyperbolic manifolds in detail.

\subsection{The hyperbolic four-manifolds}
A cubulation $C$ is an isometric pairing of the facets of $n$ hypercubes $H_1,\ldots, H_n$. Every hypercube has $8$ cubic facets and $24$ square two-dimensional faces. Every square face is contained in exactly two cubic facets and is hence identified by the pairing to two other square faces (counted with multiplicities).
 
Consider the abstract set of $24 n$ square faces and connect two of them if they are identified by some pairing: the resulting graph will be a union of \emph{cycles}. The number $k$ of the resulting cycles and the length of each depend not only on $n$, but on the combinatorial structure of the cubulation as well.

We have described a procedure that transforms the cubulation $C$ into a hyperbolic orientable four-manifold $M$. A considerable amount of information about $M$ can be derived directly from $C$, thanks to the following correspondences that follow immediately from the construction:
\begin{itemize}
\item the $n$ hypercubes in $C$ correspond to the $n$ copies of the block $\block$ in $M$; 
\item the $4n$ pairs of cubes in $C$ correspond to the totally geodesic octahedral $3$-manifolds separating two adjacent blocks;
\item the $k$ cycles of squares in $C$ correspond to the $k$ cusps of $M$.
\end{itemize}

For example, the volume formula is a direct consequence:
\begin{prop}\label{volume:prop}
We have $\chi(M) = 4n$ and $\Vol (M) = \frac {4\pi^2}3\chi (M) = \frac {16 n}3 \pi^2.$
\end{prop}
\begin{proof}
The manifold $M$ is tessellated by $n$ isometric copies of $\block$ and hence by $4n$ copies of the $24$-cell $\cell$, which has the volume $\frac 43 \pi^2$, see for instance \cite{K}. The formula $\Vol(M) = \frac {4\pi^2}3\chi (M)$ holds for any hyperbolic four-manifold \cite{G}. 
\end{proof}
Note that there exist cusped hyperbolic four-manifolds having any positive Euler characteristic, c.f.~\cite{RT}. 

We now turn to the $k$ cusps of $M$, corresponding to the $k$ cycles of squares in $C$. The toric maximal sections of the $\block_i$'s glue together to a maximal cusp section of $M$, determined only by the cubulation. It consists of $k$ components, one for each cycle of squares. As above, let $T$ denote the square torus obtained by identifying the opposite sides of a $2\times 2$ Euclidean square. We tessellate $T$ into four unit squares in the obvious way.

\begin{figure}
 \begin{center}
  \includegraphics[width = 9 cm]{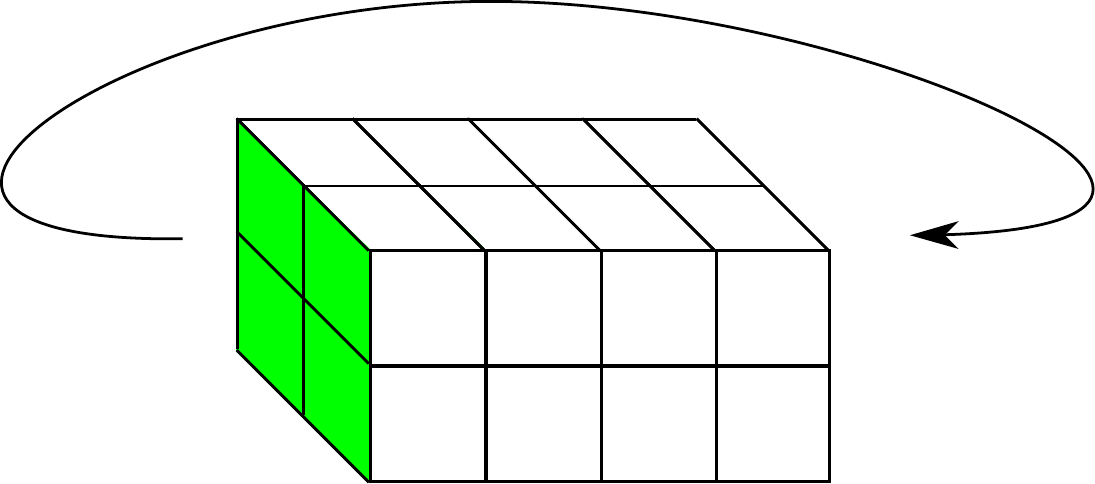}
 \end{center}
 \caption{A cusp in $M$ is constructed from $T\times [0,h]$ by glueing the two boundary tori via an isometry that preserves the tessellation into four unit squares. In the picture, $h=4$.}
 \label{cusp:fig}
\end{figure}

\begin{prop} \label{h:prop}
Let $X$ be a connected component of the maximal cusp section of $M$, corresponding to a cycle of $h$ square $2$-faces. The flat $3$-manifold $X$ is isometric to
$$T\times [0, h]/_\psi$$ 
where $\psi$ identifies $T\times 0$ and $T\times h$ via an orientation-preserving isometry of $T$ which preserves its tessellation into unit squares. Topologically, the cusp section is homeomorphic to the torus bundle over $S^1$ with one of the following monodromies:
$$\begin{pmatrix} 1 & 0 \\ 0 & 1 \end{pmatrix}, \qquad \begin{pmatrix} -1 & 0 \\ 0 & -1 \end{pmatrix}, 
\qquad \begin{pmatrix} 0 & 1 \\ -1 & 0 \end{pmatrix}.
$$
\end{prop}
\begin{proof}
The flat manifold $X$ is made of $h$ pieces, each isometric to a toric cusp $T \times [0,1]$ corresponding to some vertex in some block $\block_i$. The tessellation into squares of $T$ corresponds to the subdivision of $\block$ into $\cell_{ij}$'s. By construction, every piece is glued to the subsequent one via an isometry that preserves the tessellation: therefore the maximal cusp section looks exactly as required, see Fig.~\ref{cusp:fig}.

The manifold $X$ is homeomorphic to a torus bundle over $S^1$ with some monodromy $A$ having $\det A = 1$. Since $\psi$ preserves the tessellation, it also preserves the pair (meridian, longitude) up to signs and permutation of its components. Therefore $A$ preserves the unordered pair of coordinate axis in $\matR^2$ and hence is one of the following rotation matrices:
$$\begin{pmatrix} 1 & 0 \\ 0 & 1 \end{pmatrix}, \quad \begin{pmatrix} -1 & 0 \\ 0 & -1 \end{pmatrix}, 
\quad \begin{pmatrix} 0 & 1 \\ -1 & 0 \end{pmatrix}, \quad \begin{pmatrix} 0 & -1 \\ 1 & 0 \end{pmatrix}.
$$
The latter two are conjugate and hence give rise to the same fibred $3$-manifold, up to homeomorphism.
\end{proof}

\begin{cor}
The flat $3$-manifold $X$ has volume $4h$. The maximal cusp section of $M$ has total volume $4\times 24n$ = $96n$.
\end{cor}

Recall that a cusped hyperbolic manifold has an \emph{Epstein-Penner canonical decomposition} into geodesic ideal polytopes, determined by the choice of a section at each cusp \cite{E}. 

\begin{prop} \label{EP:prop}
The Epstein-Penner canonical decomposition of $M$ determined by the maximal cusp section is the decomposition of $M$ into $4n$ ideal $24$-cells.
\end{prop}
\begin{proof}
The maximal cusp section and the decomposition into ideal $24$-cells lift to the tessellation of $\matH^4$ by ideal $24$-cells, together with a horocusp at each ideal vertex. The set of horocusps is invariant under the action of the isometry group of the tessellation. The Epstein-Penner decomposition is formed by interpreting the horocusps as points of the light cone in $\matR^{4,1}$ and taking their convex hull. Thus, by symmetry, the resulting decomposition is just the original decomposition into $24$-cells.
\end{proof}

How can we determine the precise isometry or homeomorphism class of $X$ by looking only at the combinatorics of $C$? It is much easier to answer this question when the cubulation is \textit{orientable}.

\subsection{Orientable cubulations}\label{orientablecub:section}
Recall that a cubulation $C$ is an isometric pairing of the $8n$ facets of $n$ copies $H_1,\ldots, H_n$ of the hypercube $H$. As usual, the natural orientation of $H$ induces an orientation on all its facets, which is replicated in the copies $H_1,\ldots, H_n$. 
\begin{defn} The cubulation $C$ is \emph{orientable} if all the isometric pairings are orientation-reversing. 
\end{defn}

In our construction, a cubulation $C$ determines a hyperbolic manifold $M$ only up to mutations. When $C$ is orientable, we may resolve this ambiguity as follows: by Lemma \ref{orientation:lemma} every glueing map $\varphi_*^{11}$ is orientation-reversing, and hence we choose $\varphi_* = \varphi_*^{11}$ in all our pairings. 

This choice turns out to be very convenient for the analysis of the maximal cusp section. Recall that each cusp of $M$ corresponds to a cycle of squares in $C$, which may be represented as
\begin{equation}
\label{Q:eqn}
\xymatrix{ Q_{1} \ar@{->}[r]^{\psi_1}  & Q_{2} \ar@{->}[r]^{\psi_2} &\ \cdots\ \ar@{->}[r]^{\psi_{h-1}} &
Q_{h} \ar@{->}[r]^{\psi_h} & Q_1.
}
\end{equation}

Each $Q_i$ is a square in some hypercube and each $\psi_i$ is an isometry. The composition $\psi = \psi_h\circ \cdots \circ \psi_1$ is an isometry of a Euclidean square, whose conjugacy class depends only on the cycle and thus is called its \emph{monodromy}. Since $C$ is orientable, the monodromy $\psi$ is orientation-preserving and is represented (up to a conjugation) by one of the following rotation matrices:
$$\begin{pmatrix} 1 & 0 \\ 0 & 1 \end{pmatrix}, \qquad \begin{pmatrix} -1 & 0 \\ 0 & -1 \end{pmatrix}, 
\qquad \begin{pmatrix} 0 & 1 \\ -1 & 0 \end{pmatrix}.
$$
As above, let $T$ be the torus obtained by identifying the opposite faces of the $2\times 2$ square $[-1,1]^2$.
The matrix $\psi$ acts on $[-1,1]^2$ and hence on $T$.
\begin{prop} \label{monodromy:prop}
Let $C$ be an orientable cubulation. The cusp section corresponding to a cycle of $h$ squares with monodromy $\psi$ is isometric to $T \times [0,h]/_\psi$.
\end{prop}
\begin{proof}
The cycle of squares (\ref{Q:eqn}) corresponds to a cycle of tori
$$\xymatrix{ T_{1} \ar@{->}[r]^{f_1}  & T_{2} \ar@{->}[r]^{f_2} &\ \cdots\ \ar@{->}[r]^{f_{h-1}} &
T_{h} \ar@{->}[r]^{f_h} & T_1.
}$$
in the cusp section. Each $T_i$ is a square $2\times 2$ torus. Let $S_i\subset T_i$ be the unit top-left square in $T_i$: it is naturally dual to the square $Q_i$. By our convention, the isometry $f_i$ sends $S_i$ onto $S_{i+1}$, and it does so via an isometry dual to $\psi_i$. Therefore $f = f_h\circ \cdots \circ f_1$ sends $S_1$ to itself via a map conjugate to $\psi$. Therefore the whole of $f$ is conjugate to $\psi$, as required.
\end{proof}

\begin{example} \label{1:example}
Define an orientable cubulation by taking one hypercube $H$ and pairing its opposite facets via translations. All the parallel square $2$-faces are identified: we thus get $6$ cycles of $4$ squares each; each cycle has trivial monodromy. The resulting hyperbolic manifold $M$ has $6$ cusps. Each maximal cusp section is a $3$-torus isometric to a right-angled parallelepiped with side lengths $2$, $2$ and $4$ whose opposite faces are identified by translations.
\end{example}

\begin{example} \label{2:example}
Define an orientable cubulation by taking two copies $H_1$ and $H_2$ of the hypercube $H$ and pairing each facet of $H_1$ with the corresponding facet of $H_2$ via the identity map. We get $24$ cycles of square $2$-faces, each cycle containing only two squares with trivial monodromy. The resulting hyperbolic manifold $M$ has $24$ cusps. Each maximal cusp section is a cubic $3$-torus isometric to a cube with side-lengths $2$ whose opposite faces are identified by translations. This very symmetric manifold $M$ may be constructed directly by taking eight copies of the triple-coloured $24$-cell $\cell$ and glueing them together according to the following cubic diagram:
$$
\xymatrix@!0{ 
& \cell_{111} \ar@{<->}^R[rr]\ar@{<->}'[d][dd]_B 
& & \cell_{112} \ar@{<->}[dd]^B 
\\ 
\cell_{211} \ar@{<->}[ur]^G\ar@{<->}[rr]^{\ \ R}\ar@{<->}[dd]_B 
& & \cell_{212} \ar@{<->}[ur]^G\ar@{<->}[dd]^<<<{B}
\\
& \cell_{121} \ar@{<->}'[r]^R[rr]
& & \cell_{122} 
\\
\cell_{221} \ar@{<->}[rr]_R\ar@{<->}[ur]^G 
& & \cell_{222} \ar@{<->}[ur]_G 
}
$$

The manifold $M$ is analogous to the octahedral manifold $N$: indeed we can compute $\Iso(M)$ by applying the same proof of Proposition~\ref{N:prop}. Let $H$ (resp. $V$, and $L$) be the isometry of $M$ that interchanges $\mathscr{C}_{ij1}$ with $\mathscr{C}_{ij2}$ (resp. $\mathscr{C}_{i1j}$ with $\mathscr{C}_{i2j}$, and $\mathscr{C}_{1ij}$ with $\mathscr{C}_{2ij}$). Then $\iota = H\cdot V\cdot L$ is the central involution of the above cubical diagram. We get the following exact sequence:
$$0 \longrightarrow \matZ_2 \oplus \matZ_2 \oplus \matZ_2 \longrightarrow \Iso(M) \longrightarrow \Iso(\mathscr{C}) \longrightarrow 0,$$
where the group $\matZ_2 \oplus \matZ_2 \oplus \matZ_2$ is generated by $H$, $V$ and $L$. 
\end{example}

\subsection{Uniqueness}
Two cubulations 
$$C = \{H_1,\ldots, H_n\}, \qquad C' = \{H_1', \ldots, H_n'\}$$ 
are \emph{combinatorially equivalent} if there is a sequence of isometries $\{\varphi_i\colon H_i \to H_i'\}_{i=1,\ldots, n}$ which transforms all the pairings of $C$ into the pairings of $C'$. Below we prove the following theorem.

\begin{teo} \label{eq:teo}
Non-equivalent cubulations with at least $3$ hypercubes produce non-homeomorphic hyperbolic four-manifolds.
\end{teo}
A cubulation actually produces a finite set of hyperbolic four-manifolds related by mutations, and the theorem says that any two manifolds produced by non-equivalent cubulations with at least $3$ hypercubes are non-homeomorphic. A similar theorem was proved in \cite{CFMP}, and our proof strategy is the same. We do not know if the hypothesis on the number of hypercubes is necessary, but it helps to simplify the arguments. We start by proving a lemma.

\begin{lemma} \label{24:lemma}
Combinatorially non-equivalent cubulations with at least $3$ hypercubes produce combinatorially non-equivalent decompositions into ideal $24$-cells.
\end{lemma}
\begin{proof}
A cubulation $C$ consisting of $n$ hypercubes gives rise to a hyperbolic four-manifold $M$ (defined up to mutations) which can be decomposed into $4n$ hyperbolic ideal regular $24$-cells. We show that the cubulation $C$ can be recovered (up to combinatorial equivalence) from such a decomposition: this proves the lemma.

Recall that in the block $\block$ every $24$-cell is adjacent to two other $24$-cells along two sets of eight facets sharing same colour. If every $24$-cell of the decomposition is adjacent only to two other $24$-cells along eight facets, then the blocks can be recovered from the decomposition into $24$-cells. Thus, the decomposition into blocks determine $C$ and the lemma is proved.

If not, there is a $24$-cell which is adjacent to three $24$-cells, along eight facets to each. This implies that the block $\block$ that contains this $24$-cell is incident along all of its $8$ geodesic boundary components to another block $\block'$ (which might coincide with $\block$). Therefore $C$ consists only of one or two hypercubes, which contradicts our hypothesis.
\end{proof}

We would like to conclude by saying that the decomposition into ideal $24$-cells is determined by the topology of $M$ only. Proposition \ref{EP:prop} says that the decomposition is determined by the topology of $M$ \emph{and} the maximum cusp section: thus we need to prove that the maximum cusp section is determined by the topology. When $M$ has only one cusp this is immediate: when the are more cusps the situation is more delicate, because the maximum cusp section depends on the ratios of the volumes spanned by different cusp sections. Luckily, the topology alone tells us which ratio to choose.
A similar argument was used in \cite{CFMP}.

\begin{lemma} \label{h:lemma}
Let $M$ be a hyperbolic manifold produced by a cubulation. Then the maximal cusp section  of $M$ is determined only by its topology.
\end{lemma}
\begin{proof}
The maximal cusp section is obtained by selecting for each cusp corresponding to a cycle of $h$ squares the unique section $X$ with three-dimensional volume equal to $4h$. If we show that the integer $h$ can be recovered intrinsically from the topology of the cusp, the proof is finished.

More precisely, the flat manifold $X$ is only determined up to similarity, and we now prove that $h$ can be recovered from its similarity class. As Proposition \ref{h:prop} says, up to a similarity we have $X=T\times [0,h]/_\psi$, where $T$ is the $2\times 2$ square torus and $\psi$ identifies $T\times 0$ with $T\times h$ via an orientation-preserving isometry of $T$ that preserves its tessellation into four unit squares. 

We have that $X = \matR^3/_\Gamma$ where $\Gamma$ is a discrete group of orientation-preserving isometries of $\matR^3$ without fixed points, determined only up to a similarity. The group $\Gamma$ contains a finite-index translation subgroup $\mathcal{T}< \Gamma$ isomorphic to $\matZ^3$ which may be seen as a lattice in $\matR^3$, which is also defined up to a similarity. 

Let $v_1, v_2 \in \mathcal{T}$ be two vectors such that the following conditions are satisfied:
\begin{enumerate}
\item $v_1$ and $v_2$ are orthogonal and have the same length $l$;
\item the number $h=\frac{2\Vol (X)} {l^3}$ is an integer;
\item $v_1$ and $v_2$ are the shortest such vectors.
\end{enumerate}
The resulting integer $h=\frac{2\Vol (X)}{l^3}$ depends only on the similarity class of $X$ because the formula is invariant under a rescaling of the flat metric on $X$. To show that in this way we recover the number $h$, we have to analyse separately a number of cases proving the existence of the vectors $v_1$ and $v_2$ in each of them. Recall that up to a similarity we have $X = T\times [0,h]/_\psi$. Then, if $h\geqslant 3$, it is clear that $v_1$ and $v_2$ are the length two vectors corresponding to the translations along the respective dimensions of $T$, which is a side of the parallelepiped that gives rise to the lattice $\mathcal{T}$. If $h=2$, there can be another pair of vectors having length two, not necessarily corresponding to $T$, that we can choose as $v_i$'s. However, we always have $l=2$ and $\Vol(X) = l^2h = \frac{l^3h}2$. Thus, we are left with the case of $h=1$.

Suppose $\psi$ is not a translation, that is $X$ is not a $3$-torus. Thus $\mathcal{T}\neq \Gamma$ and $\mathcal{T}$ is generated by $(2,0,0), (0,2,0),$ and a third vector of the form $(a,b,c)$ with $c=2$ or $c=4$, which corresponds to either $\psi^2$ or $\psi^4$ depending on the homeomorphism class of $X$. In this case we conclude as above. Finally, we are left with the case $\mathcal{T}=\Gamma$ and $h=1$. We know that $\psi$ preserves the tessellation of the torus $T$ into four squares. There are four possible such translations, and so the lattice $\mathcal{T}$ is generated by one of the following triples: 
$$ (2,0,0), (0,2,0), (0,0,h); $$
$$ (2,0,0), (0,2,0), (1,0,h); $$
$$ (2,0,0), (0,2,0), (0,1,h); $$
$$ (2,0,0), (0,2,0), (1,1,h). $$
It is easy to check that in all the above cases $v_1$ and $v_2$ are either the original vectors $(2,0,0)$ and $(0,2,0)$ that generate the torus $T$, or some other pair of vectors having length $l=2$, that we can choose as $v_i$'s. The requirement that $\frac{2\Vol(X)}{l^3}$ is an integer excludes the solutions with $l=\sqrt 2$ that arise for instance if $v_1=(1,0,1)$ and $v_2=(-1,0,1)$ when $\Gamma$ is generated by $(2,0,0)$, $(0,2,0)$, and $(1,0,1)$.
\end{proof}

\prf{eq:teo}
Non-equivalent cubulations produce combinatorially non-equivalent decompositions into ideal $24$-cells by Lemma \ref{24:lemma}. These decompositions are certain Epstein-Penner canonical decompositions by Proposition \ref{EP:prop}, which depend only on the topology of the manifolds by Lemma \ref{h:lemma}. Therefore, non-equivalent decompositions give rise to non-homeomorphic manifolds.
\prfend

\section{Four-manifolds with $k$ cusps} \label{manifolds:section}
We can now use cubulations to construct plenty of cusped hyperbolic four-manifolds. Let $\rho_k(V)$ denote the number of non-homeomorphic orientable hyperbolic four-manifolds with $k$ cusps and volume at most $V$. Below, we prove the following theorem.
\begin{teo} \label{V:teo}
For every integer $k\geqslant 1$ there are two constants $C>1$, $V_0>0$ such that $\rho_k(V) > C^{V\ln V}$ for all $V>V_0$.
\end{teo}

Let us start by constructing a hyperbolic four-manifold with one cusp.

\subsection{A hyperbolic manifold with one cusp.}  \label{manifold:subsection}
In the previous section we have shown how to transform a four-dimensional cubulation into a hyperbolic four-manifold. Recall that the $2$-dimensional faces in a cubulation are partitioned into cycles, and that each cycle gives rise to a cusp in the respective manifold. Therefore, in order to construct a hyperbolic manifold with one cusp we have to find a cubulation with only one cycle of square $2$-faces. We construct such a cubulation here.

\begin{figure}
 \begin{center}
  \includegraphics[width = 8cm]{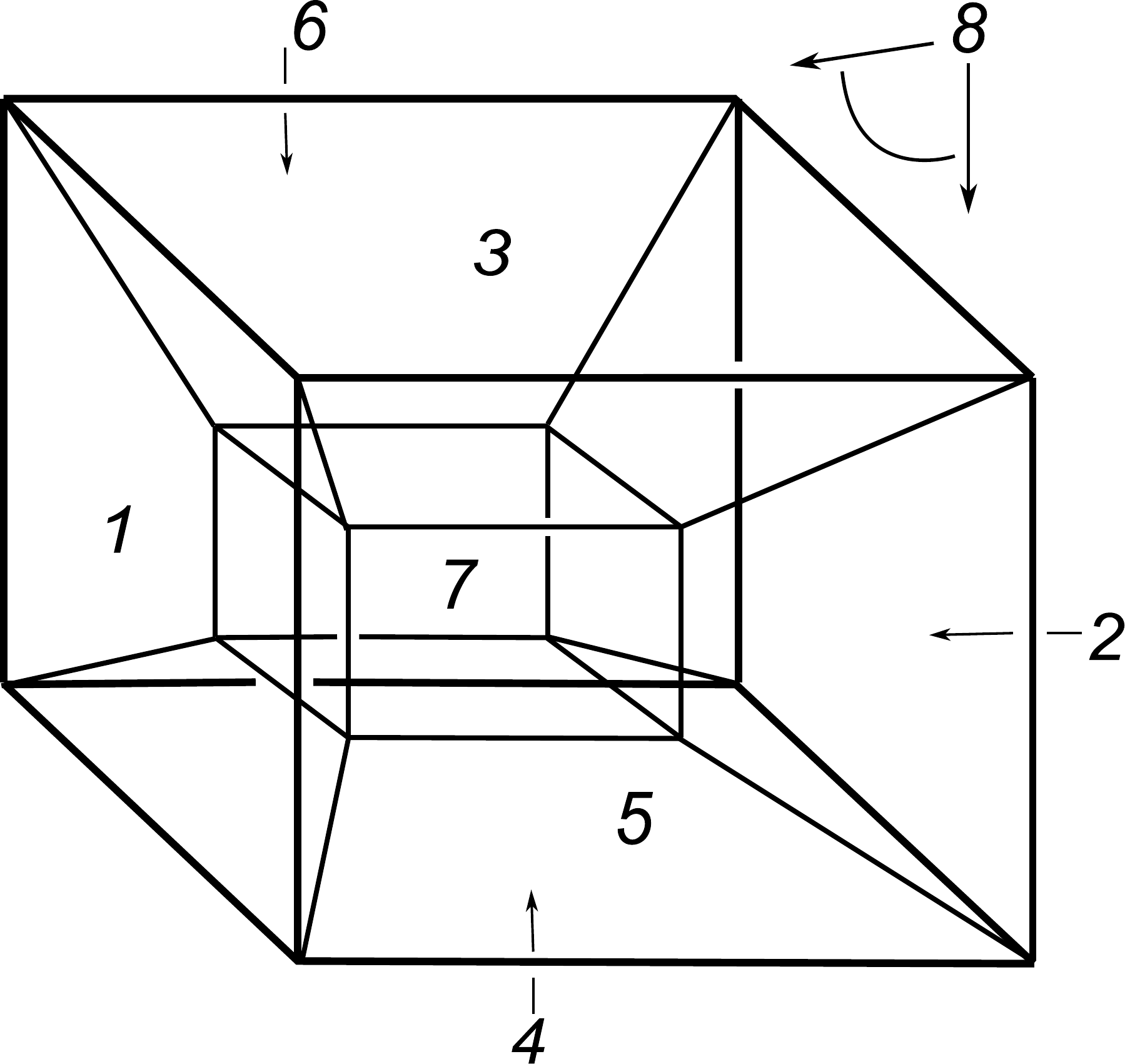}
 \end{center}
 \caption{The Schlegel diagram of the hypercube $H$. Each facet is labelled with an integer $1$ to $8$. The innermost facet has label $7$, the outermost has label $8$.}
 \label{cube:fig}
\end{figure}

The cubulation consists of one hypercube $H$ shown in Fig.~\ref{cube:fig}. Its eight cubic facets are numbered from $1$ to $8$. We pair the opposite faces $(1,2)$, $(3,4)$, $(5,6)$ and $(7,8)$ via certain Euclidean isometries. The isometries are described in Fig.~\ref{cubulationM:fig}: we have reproduced every facet from Fig.~\ref{cube:fig} and drawn a frame $x,y,z$ in each, that consists of a vertex and three adjacent oriented edges ordered as $x,y,z$. We identify the facets on the left with the facets on the corresponding right using the unique isometry that matches the frames.

\begin{figure}
 \begin{center}
  \subfigure[]{\includegraphics[width = 7cm]{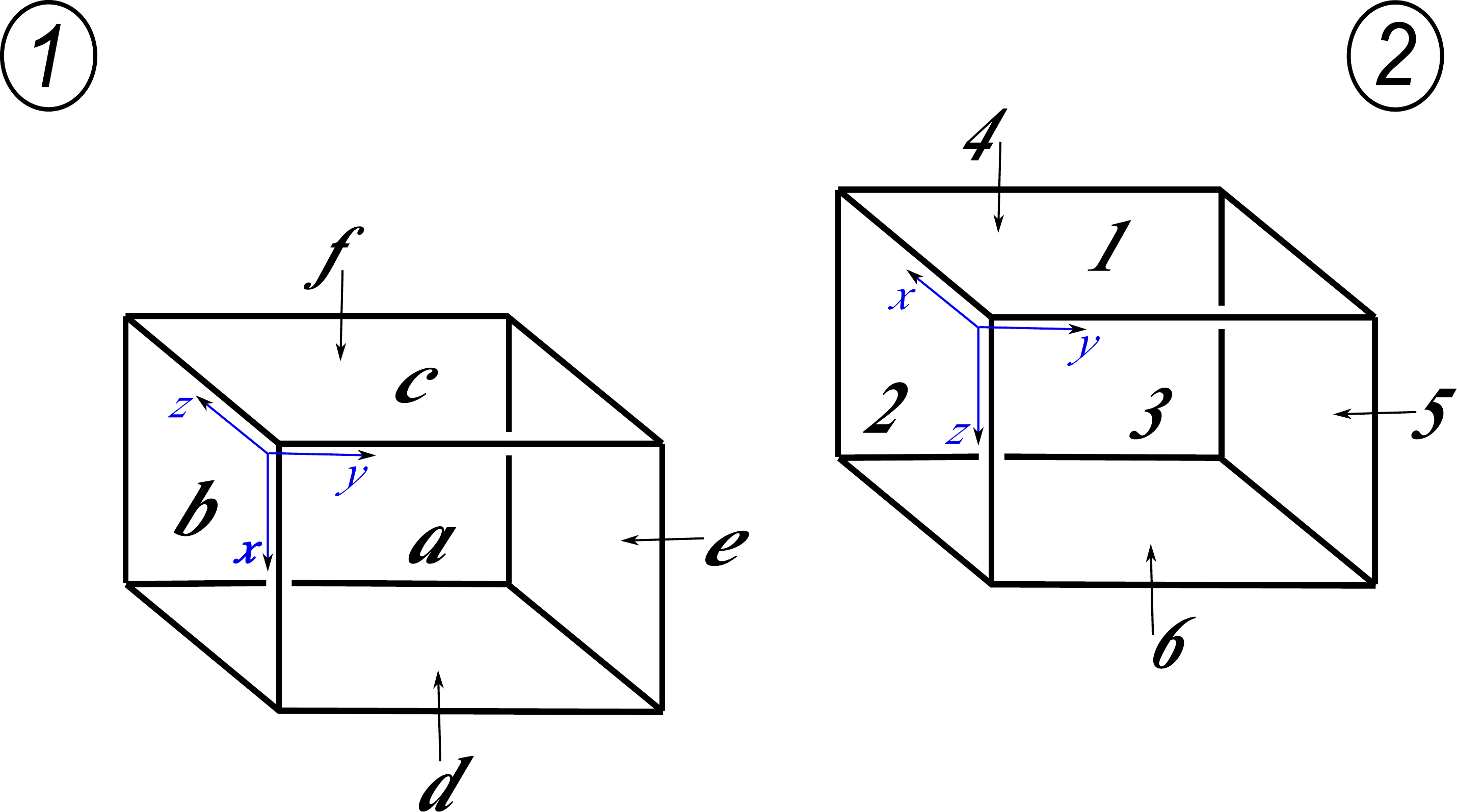}}
  \subfigure[]{\includegraphics[width = 7cm]{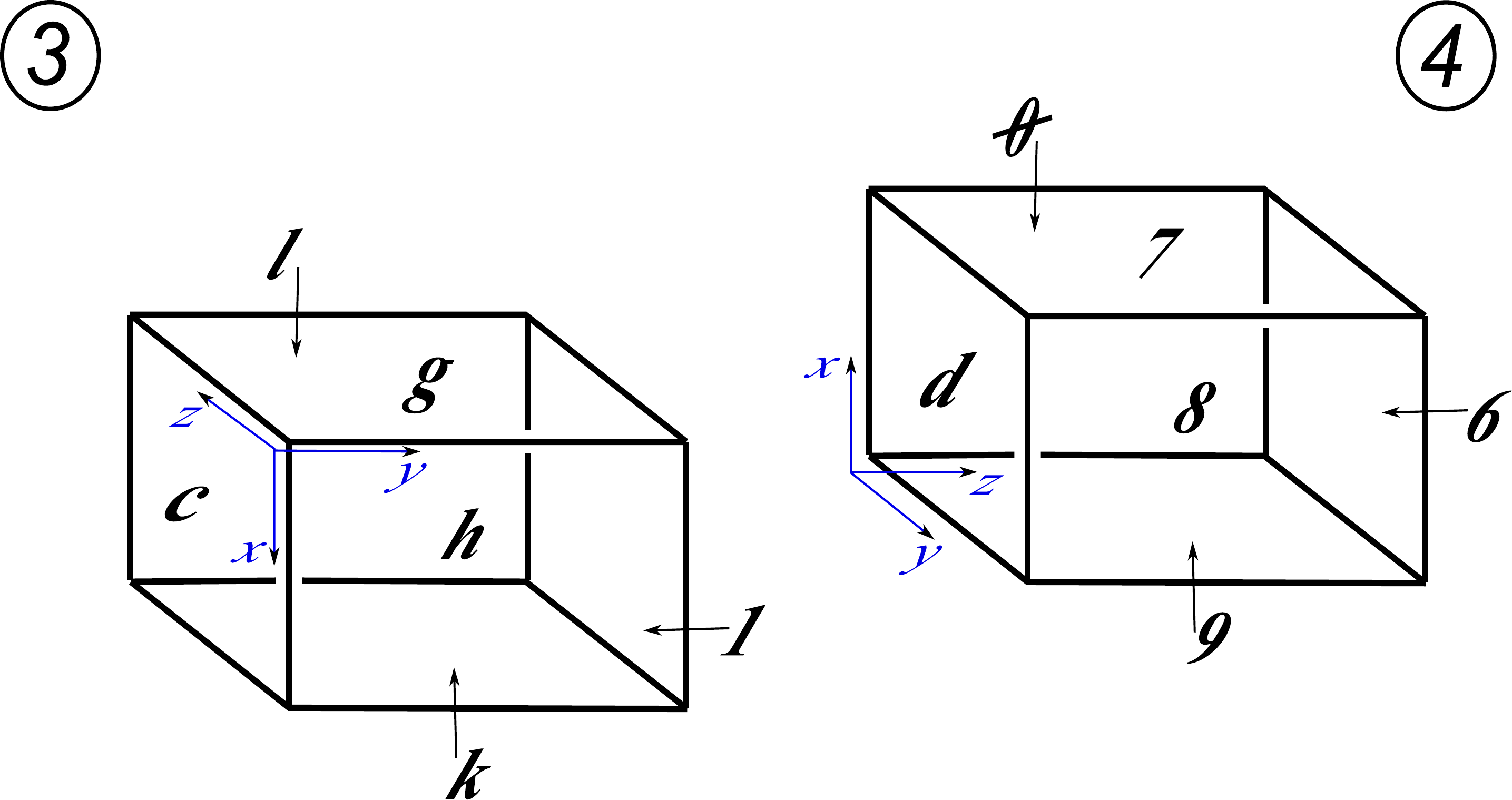}}
  \subfigure[]{\includegraphics[width = 7cm]{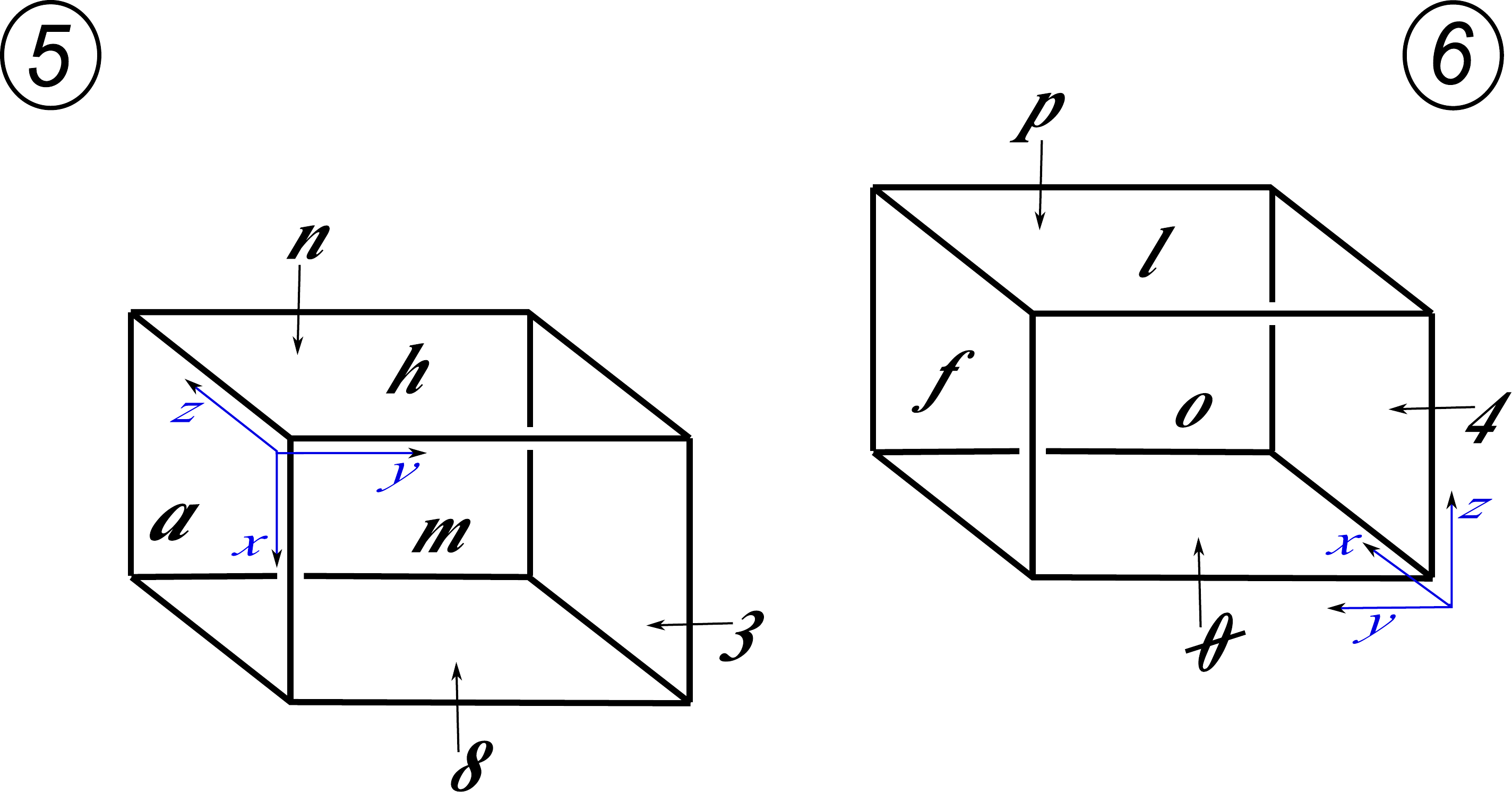}}
  \subfigure[]{\includegraphics[width = 7cm]{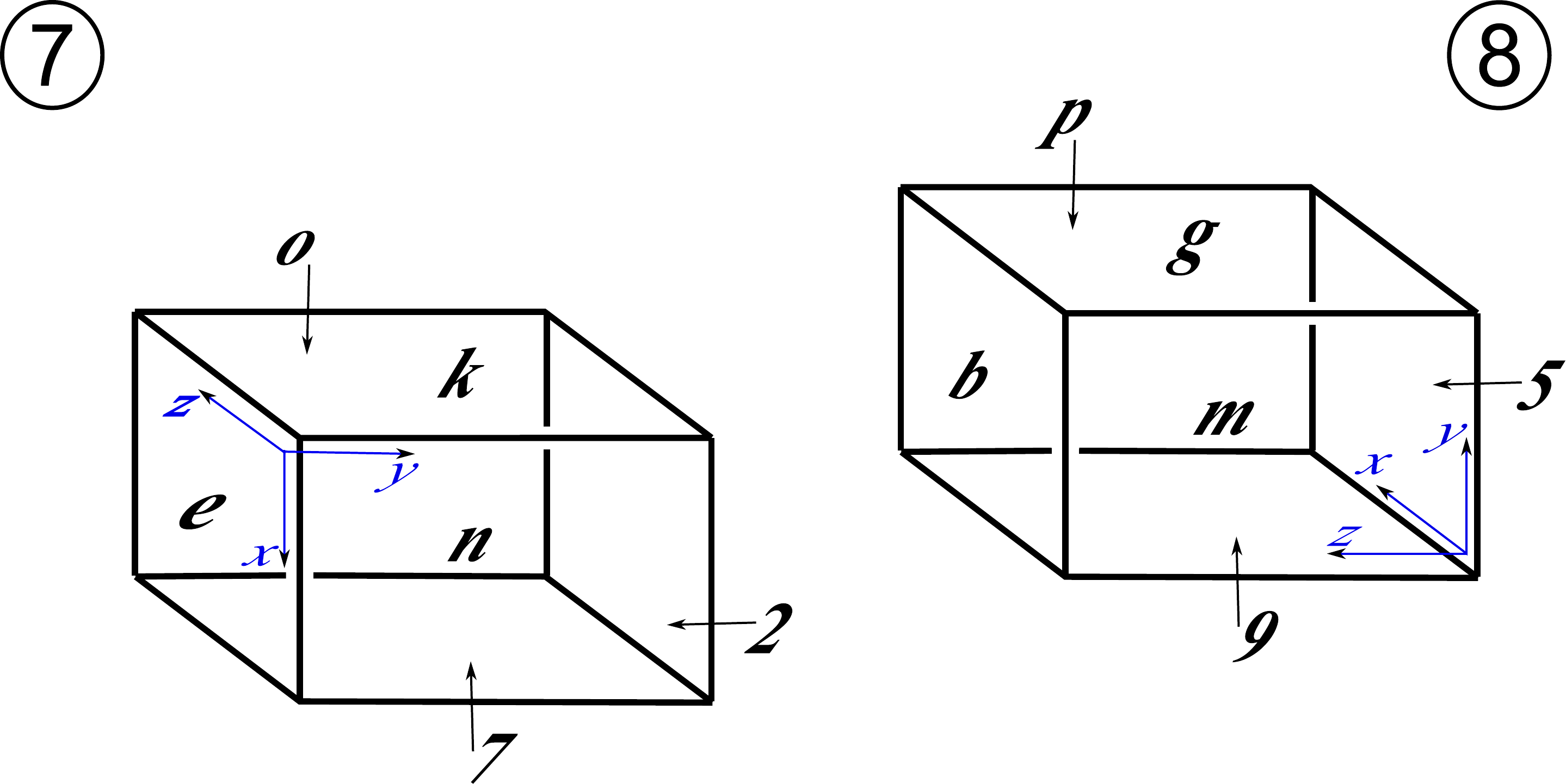}}
 \end{center}
 \caption{The pairing of the facets of $H$. Every row describes a pair of opposite facets, which are identified by the unique isometry which matches the $x,y,z$ frames.
The facet labels of $H$ are encircled. The square $2$-face labels are bold. All facets have the same orientation, except for $8$.}
 \label{cubulationM:fig}
\end{figure}
 
\begin{prop} The cubulation is orientable and has a unique cycle of $2$-faces.
\end{prop}
\begin{proof} 
Every facet inherits an orientation from the orientation of $H$. In Fig.~\ref{cube:fig}, all the facets are projected inside $8$. Therefore, the orientation of the facets $1$ to $7$ is the same, and that of $8$ is reversed. The isometries in Fig.~\ref{cubulationM:fig} pairing $(1,2)$, $(3,4)$ and $(5,6)$ are orientation-reversing as isometries in $\matR^3$, while that pairing $(7,8)$ is orientation-preserving. Thus, all the isometries are orientation-reversing if seen with the intrinsic orientation of $\partial H$: the cubulation is hence orientable.

We now prove that all the square $2$-faces form a unique cycle. We have labelled the $24$ distinct $2$-faces of $H$ with the symbols 
$$\mathbf a, \mathbf b, \mathbf c, \mathbf d, \mathbf e, \mathbf f, \mathbf 1, \mathbf 2, \mathbf 3, \mathbf 4, \mathbf 5, \mathbf 6, \mathbf g, \mathbf h, \mathbf k, \mathbf l, \mathbf 7, \mathbf 8, \mathbf 9, \mathbf 0, \mathbf m, \mathbf n, \mathbf o, \mathbf p.$$
The labels are shown in Fig.~\ref{cubulationM:fig}. The facet pairing induces the following identifications of the corresponding $2$-faces: 
\begin{align}\label{cubulationM:glueing1}
\begin{array}{cccc}
\mathbf{a}\rightarrow \mathbf{1}& \mathbf{f}\rightarrow \mathbf{6}& \mathbf{1}\rightarrow \mathbf{8}& \mathbf{c}\rightarrow \mathbf{\emptyset}\\
\mathbf{b}\rightarrow \mathbf{2}& \mathbf{e}\rightarrow \mathbf{5}& \mathbf{g}\rightarrow \mathbf{9}& \mathbf{k}\rightarrow \mathbf{7}\\
\mathbf{c}\rightarrow \mathbf{3}& \mathbf{d}\rightarrow \mathbf{4}& \mathbf{h}\rightarrow \mathbf{d}&
\mathbf{l}\rightarrow \mathbf{6}
\end{array}
\end{align}
\begin{align}\label{cubulationM:glueing2}
\begin{array}{cccc}
\mathbf{m}\rightarrow \mathbf{\emptyset}& \mathbf{n}\rightarrow \mathbf{l}& \mathbf{e}\rightarrow \mathbf{9}& \mathbf{2}\rightarrow \mathbf{g}\\
\mathbf{3}\rightarrow \mathbf{f}& \mathbf{a}\rightarrow \mathbf{4}& \mathbf{n}\rightarrow \mathbf{5}& \mathbf{o}\rightarrow \mathbf{b}\\
\mathbf{h}\rightarrow \mathbf{o}& \mathbf{8}\rightarrow \mathbf{p}& \mathbf{k}\rightarrow \mathbf{m}&
\mathbf{7}\rightarrow \mathbf{p}
\end{array}
\end{align}

Altogether these identifications produce a unique cycle:
\begin{equation}\label{cubulationM:cycle}
\mathbf{a}\mathbf{1}\mathbf{8}\mathbf{p}\mathbf{7}\mathbf{k}\mathbf{m}\mathbf{\emptyset}
\mathbf{c}\mathbf{3}\mathbf{f}\mathbf{6}\mathbf{l}\mathbf{n}\mathbf{5}\mathbf{e}\mathbf{9}\mathbf{g}
\mathbf{2}\mathbf{b}\mathbf{o}\mathbf{h}\mathbf{d}\mathbf{4}\mathbf{a}.
\end{equation} 
\end{proof}

The algorithm described in the previous section transforms the cubulation into a cusped hyperbolic four-manifold $\mathscr M$ with one cusp.

\subsection{Proof of Theorem \ref{V:teo}}
Thanks to the discussion carried over in the previous sections, the proof of Theorem \ref{V:teo} reduces to the following lemma, which deals with cubulations only and not with hyperbolic geometry. Let $\eta_k(n)$ be the number of combinatorially distinct cubulations with at most $n$ hypercubes whose $2$-dimensional faces form exactly $k$ cycles.

\begin{lemma} \label{n:lemma}
For every integer $k\geqslant 1$ there are two constants $C, n_0>1$ such that $\eta_k(n) > C^{n\ln n}$ for all $n>n_0$.
\end{lemma}

\vspace{2pt}\noindent\textit{Proof of Theorem \ref{V:teo} from Lemma \ref{n:lemma}. }\ 
By Theorem \ref{eq:teo} non-equivalent cubulations (with at least $3$ hypercubes) produce non-homeomorphic hyperbolic four-manifolds. 
\prfend

The rest of this section is devoted to the proof of Lemma \ref{n:lemma}. The \emph{incidence graph} $G$ of a cubulation $C$ is constructed by placing a vertex for every hypercube and an edge for every facet-pairing. Every vertex in $G$ has valence $8$. We start by the following proposition.

\begin{figure}
 \begin{center}
  \includegraphics[width = 6 cm]{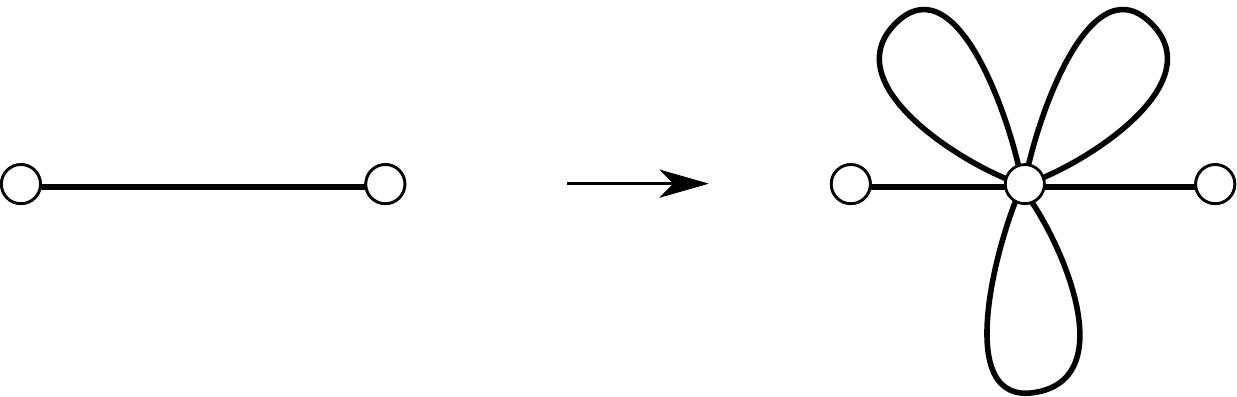}
 \end{center}
 \caption{Adding a flower to an edge of the incidence graph. The new cubulation $C'$ has one hypercube more than $C$.}
 \label{flower:fig}
\end{figure}

\begin{prop} \label{flower:prop}
Let a cubulation $C$ have $k>1$ cycles of square $2$-faces and incidence graph $G$. There is a cubulation $C'$ with $k-1$ cycles of $2$-faces, whose incident graph $G'$ is obtained from $G$ by applying the transformation depicted in Fig.~\ref{flower:fig} to some edge of $G$.
\end{prop}
\begin{proof}
Every cycle of $2$-faces in $C$ may be seen as a closed path in $G$. This closed path may pass multiple times over a vertex or edge of $G$. Let us denote these cycles by $f_1,\ldots, f_k$. 

Every edge $e$ of $G$ corresponds to a pairing of $3$-dimensional cubes, identified to a single cube $K$ in the cubulation. The edge $e$ is traversed $6$ times by the paths $f_1,\ldots, f_k$: each passing corresponds to a $2$-face of $K$. A path may traverse $e$ multiple times in both directions.

Since $k>1$, it is easy to deduce that there is at least one edge $e$ which is traversed by at least two distinct paths: in other words, not all of the six faces in the corresponding cube $K$ belong to the same cycle. Then we may pick two non-opposite $2$-faces of $K$ that belong to distinct cycles, say $f_1$ and $f_2$.

We now want to insert a new hypercube at $e$ and construct a new cubulation in which $f_1$ and $f_2$ become a single cycle. In order to do so, we construct a ``partial cubulation'' as follows. Let us consider a new hypercube $H$ as in Fig.~\ref{cube:fig} and pair the facets $(3,4)$, $(5,6)$ and $(7,8)$ as described by Fig.~\ref{cubulationM:fig}, leaving the facets $1$ and $2$ unglued. The $2$-faces of the facets $1$ and $2$ are labelled respectively with the symbols $\mathbf a, \mathbf b, \mathbf c, \mathbf d, \mathbf e, \mathbf f$ and $\mathbf 1, \mathbf 2, \mathbf 3, \mathbf 4, \mathbf 5, \mathbf 6$. The partial cubulation induces a partition of the $24$ two-dimensional faces of $H$ into six sequences, obtained by cutting the cycle (\ref{cubulationM:cycle}) at six points, separating the pairs $\mathbf a \mathbf 1$, $\mathbf c \mathbf 3$, $\mathbf f \mathbf 6$, $\mathbf 5 \mathbf e$, $\mathbf 2 \mathbf b$, $\mathbf d \mathbf 4$. These six sequences are:
$$
\mathbf{1}\mathbf{8}\mathbf{p}\mathbf{7}\mathbf{k}\mathbf{m}\mathbf{\emptyset}
\mathbf{c}, \quad 
\mathbf{3}\mathbf{f}, \quad 
\mathbf{6}\mathbf{l}\mathbf{n}\mathbf{5}, \quad 
\mathbf{e}\mathbf{9}\mathbf{g}\mathbf{2}, \quad 
\mathbf{b}\mathbf{o}\mathbf{h}\mathbf{d}, \quad 
\mathbf{4}\mathbf{a}.
$$

\begin{figure}
 \begin{center}
  \includegraphics[width = 4.5 cm]{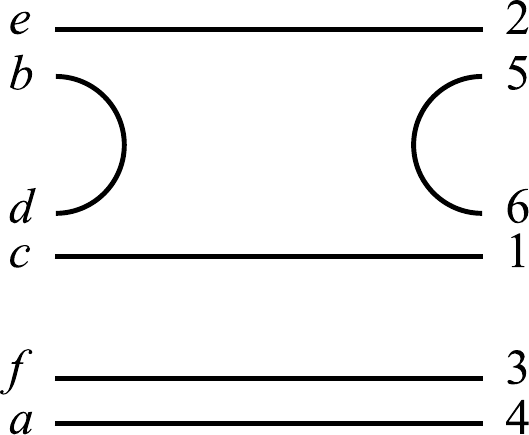}
 \end{center}
 \caption{The sequence of $2$-faces in the partial cubulation.}
 \label{braid:fig}
\end{figure}

We visualise the six paths in Fig.~\ref{braid:fig}. Each path connects a pair of $2$-faces belonging to the unglued facets $1$ and $2$. We note that there are only two paths connecting two square $2$-faces belonging to the same $3$-dimensional cube: they connect $\mathbf b$ to $\mathbf d$ and $\mathbf 5$ to $\mathbf 6$.

We now turn back to our original cubulation $C$. We visualise the six paths traversing the edge $e$ as in Fig.~\ref{braid2:fig}-(left): the three pairs of paths correspond to the three pairs of opposite faces in $K$. By hypothesis, the cycles $f_1$ and $f_2$ pass through two non-opposite faces. We now modify the cubulation $C$ by inserting the partial cubulation at the edge $e$ as suggested by Fig.~\ref{braid2:fig}. The incidence graph changes as in Fig.~\ref{flower:fig} and the cycles look now as shown in Fig.~\ref{braid2:fig}: the two cycles $f_1$ and $f_2$ are fused into a single cycle, and all other cycles remain the same. This proves the proposition.
\begin{figure}
 \begin{center}
  \includegraphics[width = 12 cm]{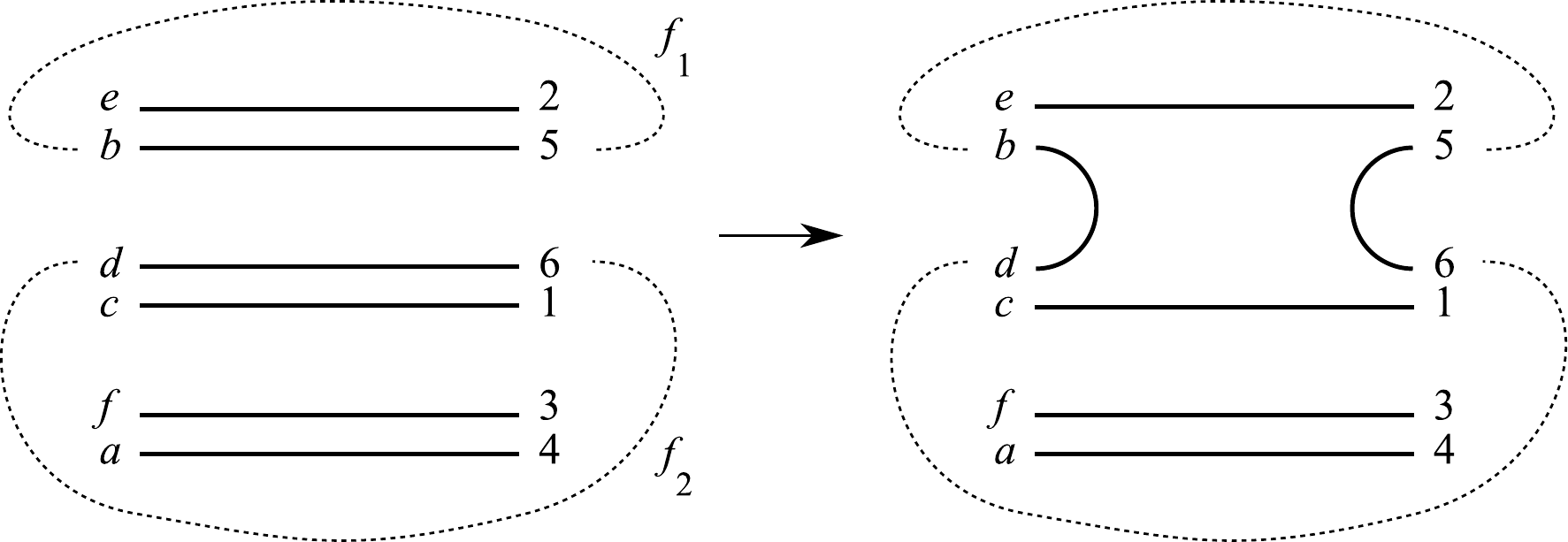}
 \end{center}
 \caption{By inserting the partial cubulation at an edge we reduce the number of cycles by one.}
 \label{braid2:fig}
\end{figure}
\end{proof}

We can now prove Lemma~\ref{n:lemma} in the case of $k=1$. 

\begin{lemma} \label{1:lemma}
There are two constants $C, n_0>1$ such that $\eta_1(n) \geqslant C^{n\ln n}$ for all $n>n_0$.
\end{lemma}
\begin{proof}
Recall that a \emph{regular $v$-graph} is a simple graph where every vertex has valency $v$. By \emph{simple} we mean that every edge connects two distinct vertices and distinct edges connect distinct pairs of vertices. 

By \cite{Bollobas}, the number of non-isomorphic $v$-regular graphs with $n$ vertices grows like $C^{n \ln n}$ for any fixed $v\geqslant 3$ and $C>1$ depending only on $v$. Therefore, it would have been sufficient to prove that every $8$-regular graph $G$ arises as the incidence graph of a cubulation with only one $2$-face. However, we are unable to prove exactly this statement, but we can use the previous proposition to prove a slightly weaker version of it, which is enough for our purposes. 

Let $G$ be a $8$-regular graph with $n$ vertices. Let $C$ be any cubulation with incidence graph $G$. The cubulation $C$ has at most $24n$ cycles of square $2$-faces: Proposition \ref{flower:prop} implies that by adding at most $24n$ flowers at its edges, as shown in Fig.~\ref{flower:fig}, we may transform $G$ into a graph $G'$, which is the incidence graph of a cubulation with only one $2$-face.

Every flower increases the number of vertices of $G$ by one. Therefore every $8$-regular graph $G$ with $n$ vertices gives rise to a graph $G'$ with at most $25n$ vertices which is the incidence graph of some cubulation with one $2$-face. The graph $G$ can be reconstructed from $G'$ by eliminating all the petals, \emph{i.e.}~the edges with coinciding endpoints: therefore non-isomorphic regular graphs $G$ and $H$ give rise to non-isomorphic graphs $G'$ and $H'$ and hence to non-isomorphic cubulations.

Let $f(n)$ be the number of non-isomorphic simple $8$-regular graphs. We know that $f(n) > {C_0}^{n\ln n}$ for some $C_0, n_0>1$ and all $n>n_0$. We have proved that $\eta_1(25 n) \geqslant f(n) > {C_0}^{n \ln n}$. The latter implies that $\eta_1(n) \geqslant {C_1}^{n \ln n}$, with some $C_0 > C_1 > 1$ and $n_1 > n_0 > 1$.
\end{proof}

We now turn to the general case, and so we will need the following proposition.

\begin{prop} \label{flower2:prop}
Let $C$ be a cubulation with incidence graph $G$. For every $k>0$ there is a cubulation $C'$ with more than $2k$ cycles of $2$-faces, whose incident graph $G'$ is obtained from $G$ by applying $k$ times the move of Fig.~\ref{flower:fig}.
\end{prop}
\begin{proof}
First we construct a partial cubulation as in the proof of Proposition \ref{flower:prop}, but this time we use the cubulation from Example~\ref{1:example}: we take a hypercube $H$ and identify three pairs of opposite facets by a translation, leaving the fourth pair unglued. In contrary to the one used in the proof of Proposition \ref{flower:prop}, this partial cubulation contains two inner cycles of $2$-faces. Take any edge $e$ of $G$. Let us insert $k$ subsequent copies of this partial cubulation at $e$: the resulting cubulation $C'$ contains more than $2k$ cycles of $2$-faces. The graph $G$ changes by adding $k$ flowers at $e$, as required.
\end{proof}

We finish the proof of Lemma \ref{n:lemma}.

\begin{lemma} 
For every integer $k > 1$ there are two constants $C, n_0>1$ such that $\eta_k(n) \geqslant C^{n\ln n}$ for all $n>n_0$.
\end{lemma}
\begin{proof}
The proof proceeds as for Lemma \ref{1:lemma}. Let $C$ a cubulation with incidence graph $G$. By adding $k$ flowers we get a cubulation $C'$ with more than $2k$ cycles of $2$-faces thanks to Proposition \ref{flower2:prop}. Then, in accordance with Proposition \ref{flower:prop}, by adding at most $25(n+k)$ more flowers we finally get a cubulation $C''$ with exactly $k$ distinct cycles of square $2$-faces. The rest of the argument is analogous to that of Lemma \ref{1:lemma}.
\end{proof}

The proof of Theorem \ref{V:teo} is now complete.

\subsection{Cusp shapes}
In Section \ref{manifold:subsection} we have introduced an orientable cubulation $C$ with one hypercube $H$ and one cycle of $2$-faces, determining a hyperbolic manifold $\mathscr{M}$ with one cusp. We are now interested in determining the topology of its cusp section.

\begin{prop}
The maximal cusp section of $\mathscr{M}$ is the flat $3$-torus obtained by identifying the opposite faces of a right-angled parallelepiped of size $2\times 2\times 24$.
\end{prop}
\begin{proof}
We need to determine the monodromy of the respective cycle of $2$-faces. In order to do so, we develop the eight cubic faces of the hypercube $H$ as in Fig.~\ref{cubulationM-monodromy:fig} and we fix a marking frame on the square face $\mathbf{a}$ in the cube $1$ as follows:
\begin{center}
\includegraphics[width= 1cm]{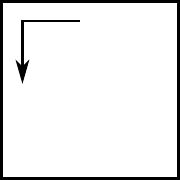}
\end{center}
Then we accurately carry the frame along the cycle 
$$\mathbf{a}\mathbf{1}\mathbf{8}\mathbf{p}\mathbf{7}\mathbf{k}\mathbf{m}\mathbf{\emptyset}
\mathbf{c}\mathbf{3}\mathbf{f}\mathbf{6}\mathbf{l}\mathbf{n}\mathbf{5}\mathbf{e}\mathbf{9}\mathbf{g}
\mathbf{2}\mathbf{b}\mathbf{o}\mathbf{h}\mathbf{d}\mathbf{4}\mathbf{a}$$
until we find out that the final frame matches the initial one (marked with a circle in the figure). Therefore the monodromy is the identity map and Proposition \ref{monodromy:prop} implies that the maximal section is as required.
\end{proof}

\begin{cor} The $3$-torus bounds geometrically a hyperbolic manifold.
\end{cor}
\begin{figure}
 \begin{center}
  \subfigure[]{\includegraphics[width = 6.2cm]{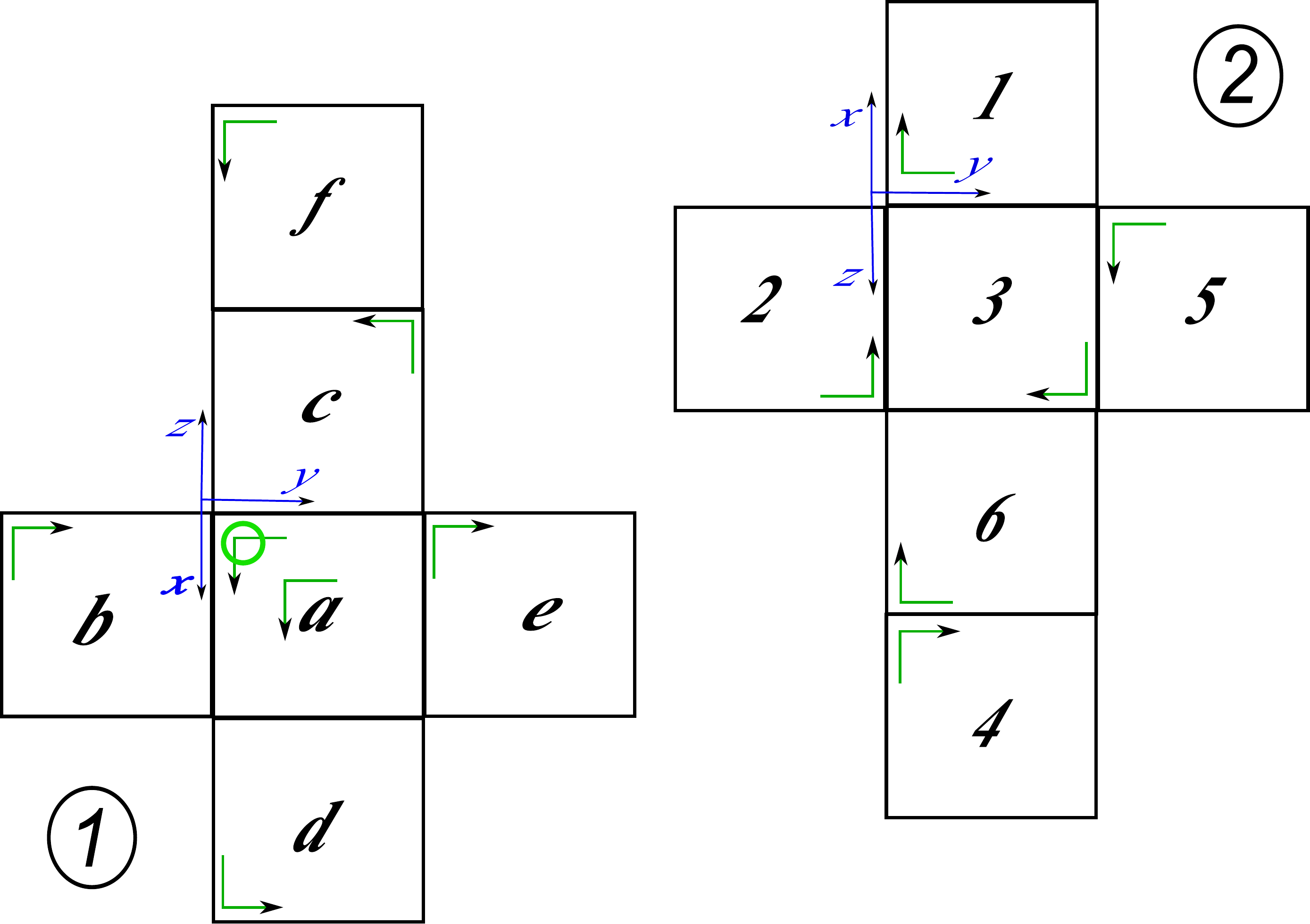}}
  \subfigure[]{\includegraphics[width = 6.2cm]{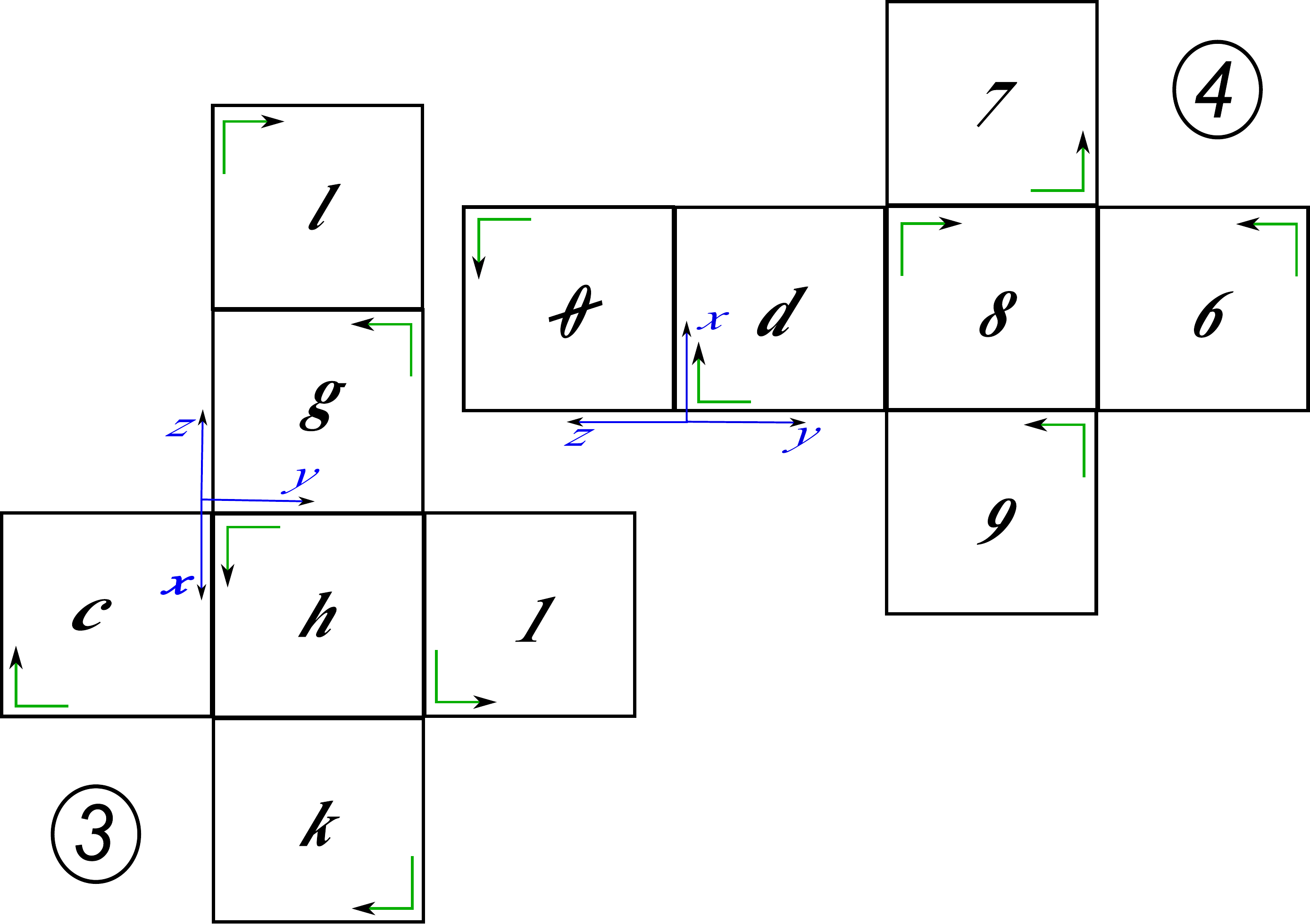}}
  \subfigure[]{\includegraphics[width = 6.2cm]{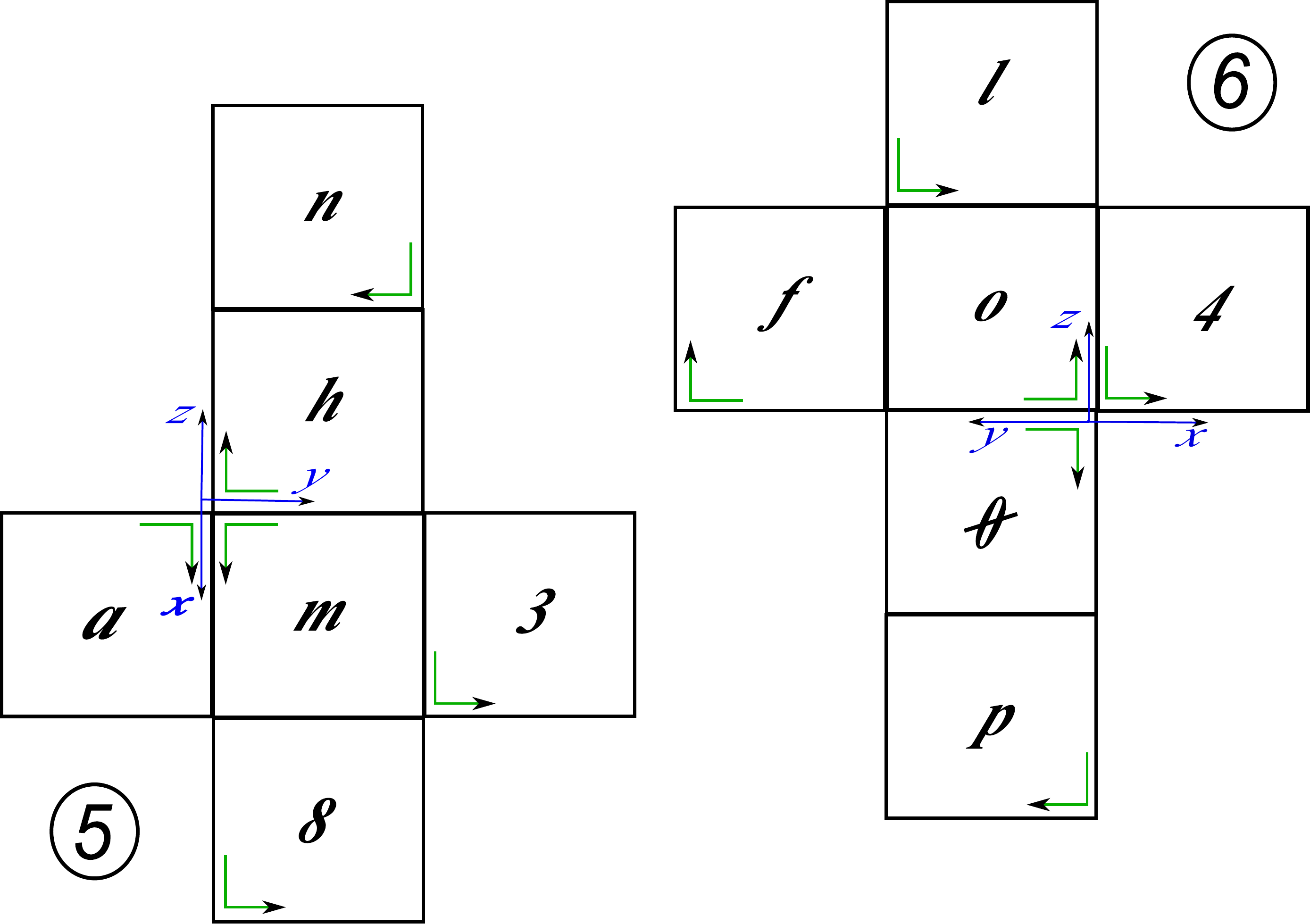}}
  \subfigure[]{\includegraphics[width = 6.2cm]{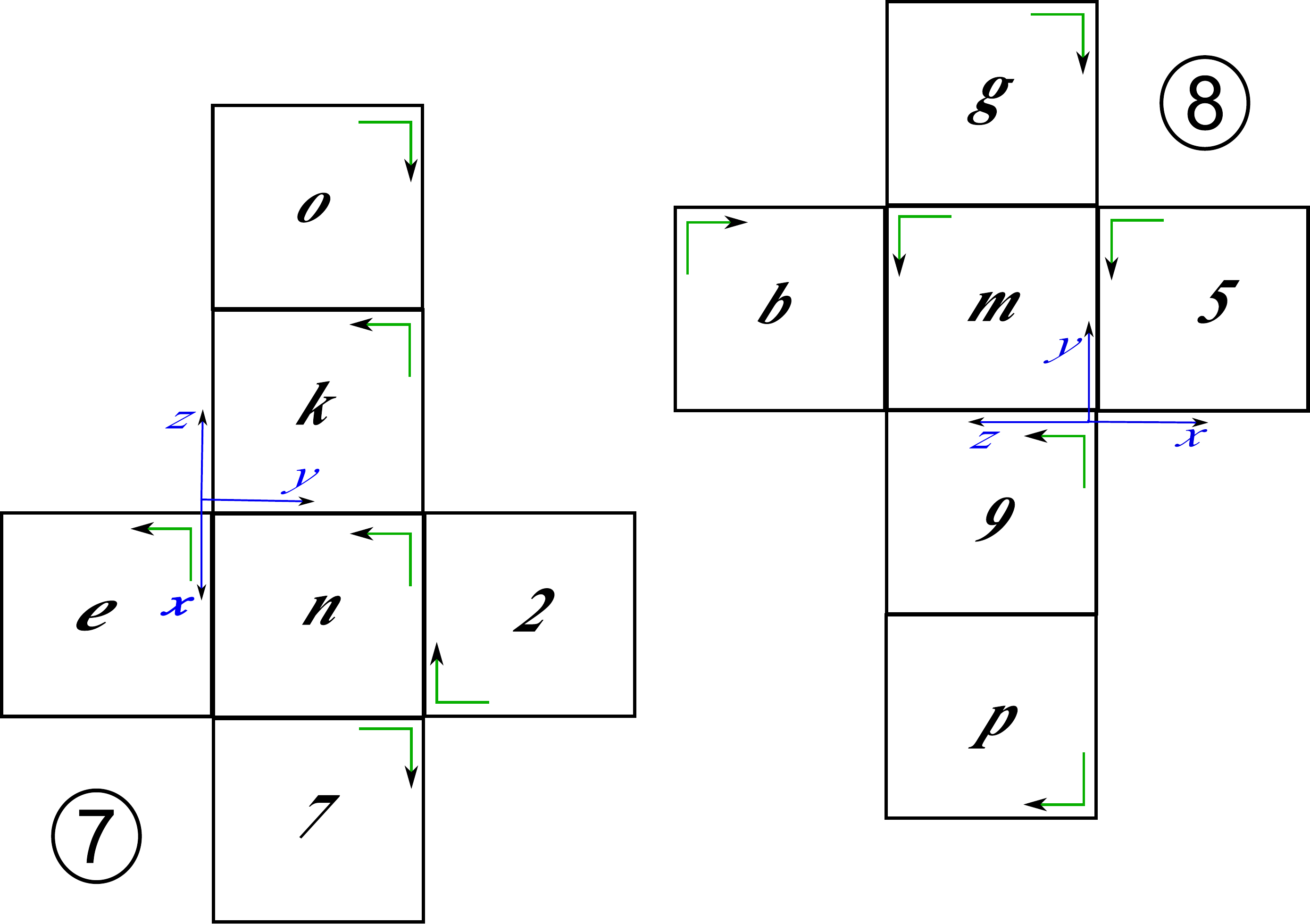}}
 \end{center}
 \caption{The pairing of the facets of $H$ and the evolution of the frame associated with the square $2$-face labelled $\mathbf{a}$. The facets of $H$ are flattened.}
 \label{cubulationM-monodromy:fig}
\end{figure}

Which flat manifolds are realisable as cusp sections of some manifolds arising from our construction? We already know from Proposition \ref{h:prop} that only three homeomorphism types may be realised. We have obtained the $3$-torus with one cusps, and we can also obtain the other two types with two cusps: 
 
\begin{prop}\label{2cusps:prop}
Let $X$ be a torus bundle over $S^1$ with monodromy 
\begin{equation*}
\left(\begin{array}{cc}
-1& 0\\
0& -1
\end{array}\right) \mbox{ or }
\left(\begin{array}{cc}
0& 1\\
-1& 0
\end{array}\right).
\end{equation*}
There exists a hyperbolic $4$-manifold with two cusps, both having a cross-section homeomorphic to $X$.
\end{prop}
\begin{proof}
We consider the orientable cubulations with one hypercube $H$ depicted in Fig.~\ref{cubulationK:fig} and Fig.~\ref{cubulationL:fig} (also, see Fig.~\ref{cube:fig}): these give rise to two hyperbolic manifolds $\mathscr{K}$ and $\mathscr{L}$. Each cubulation has two cycles of square $2$-faces: the monodromies are calculated in Fig.~\ref{cubulationK-monodromy:fig} and Fig.~\ref{cubulationL-monodromy:fig} and they are respectively
$$\left(\begin{array}{cc}
0& 1\\
-1& 0
\end{array}\right), 
\left(\begin{array}{cc}
0& -1\\
1& 0
\end{array}\right)
\quad {\rm and} \quad
\left(\begin{array}{cc}
-1& 0\\
0 & -1
\end{array}\right), 
\left(\begin{array}{cc}
-1 & 0\\
0& -1
\end{array}\right). $$
\end{proof}

\begin{figure}
 \begin{center}
  \subfigure[]{\includegraphics[width = 6.2cm]{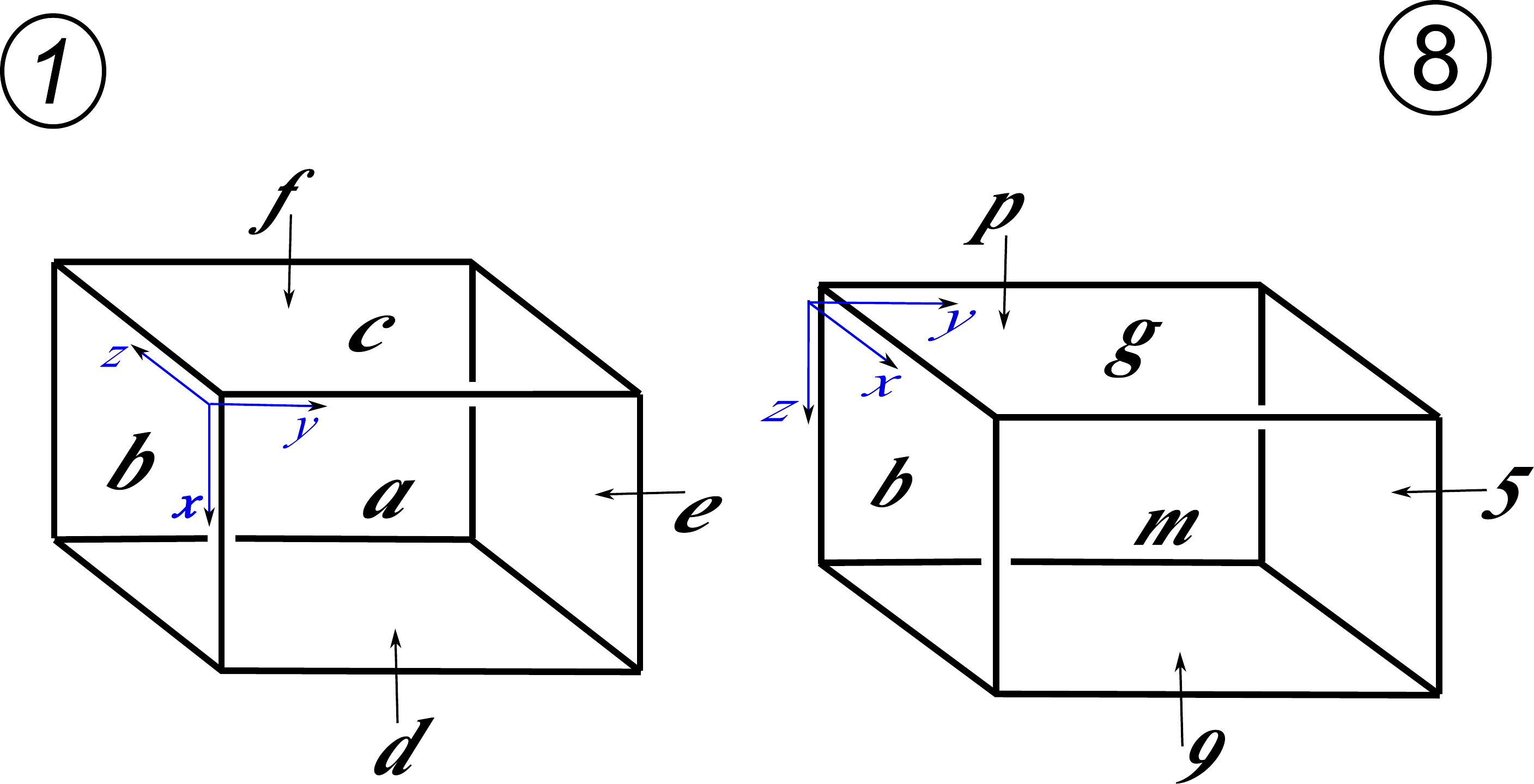}}
  \subfigure[]{\includegraphics[width = 6.2cm]{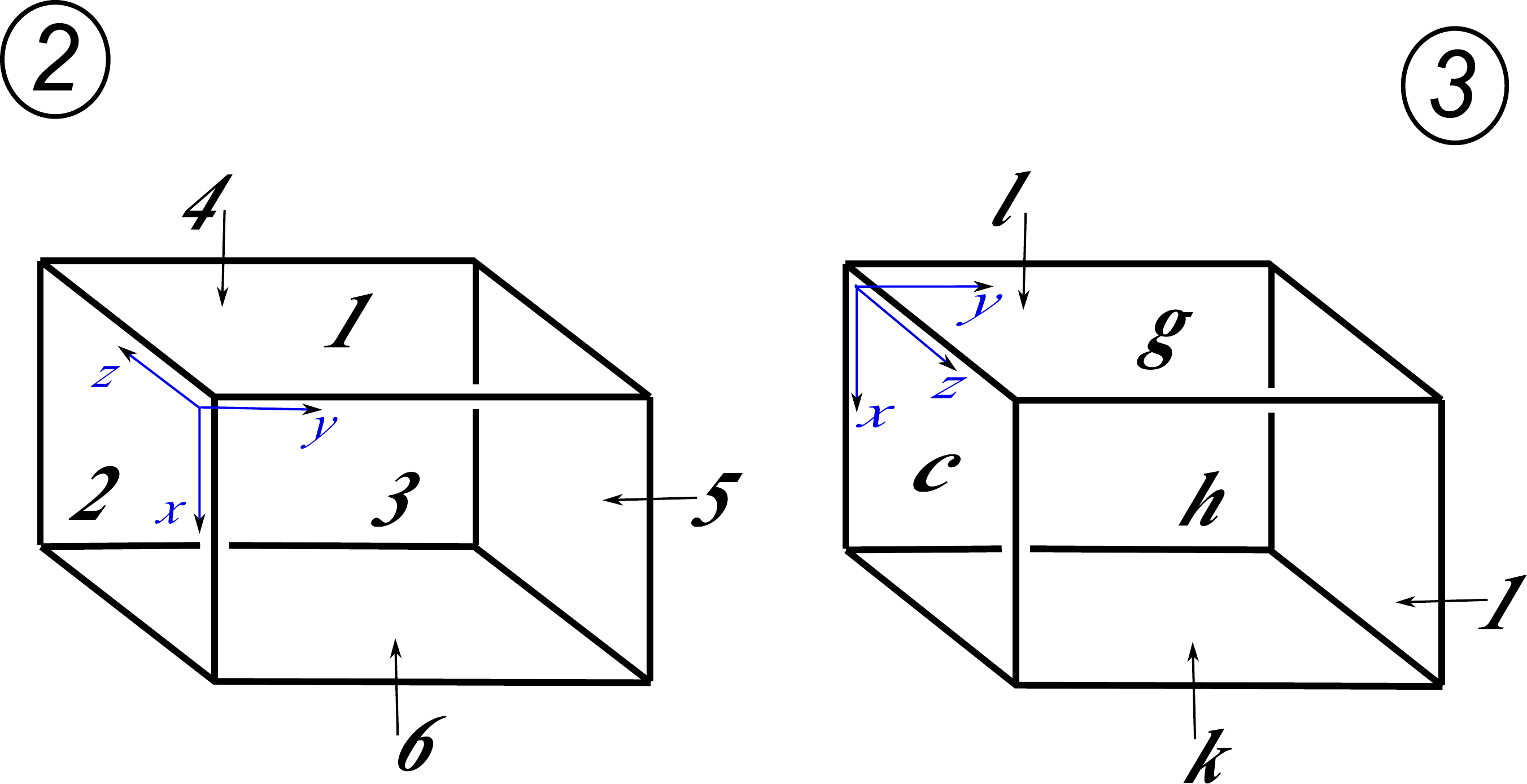}}
  \subfigure[]{\includegraphics[width = 6.2cm]{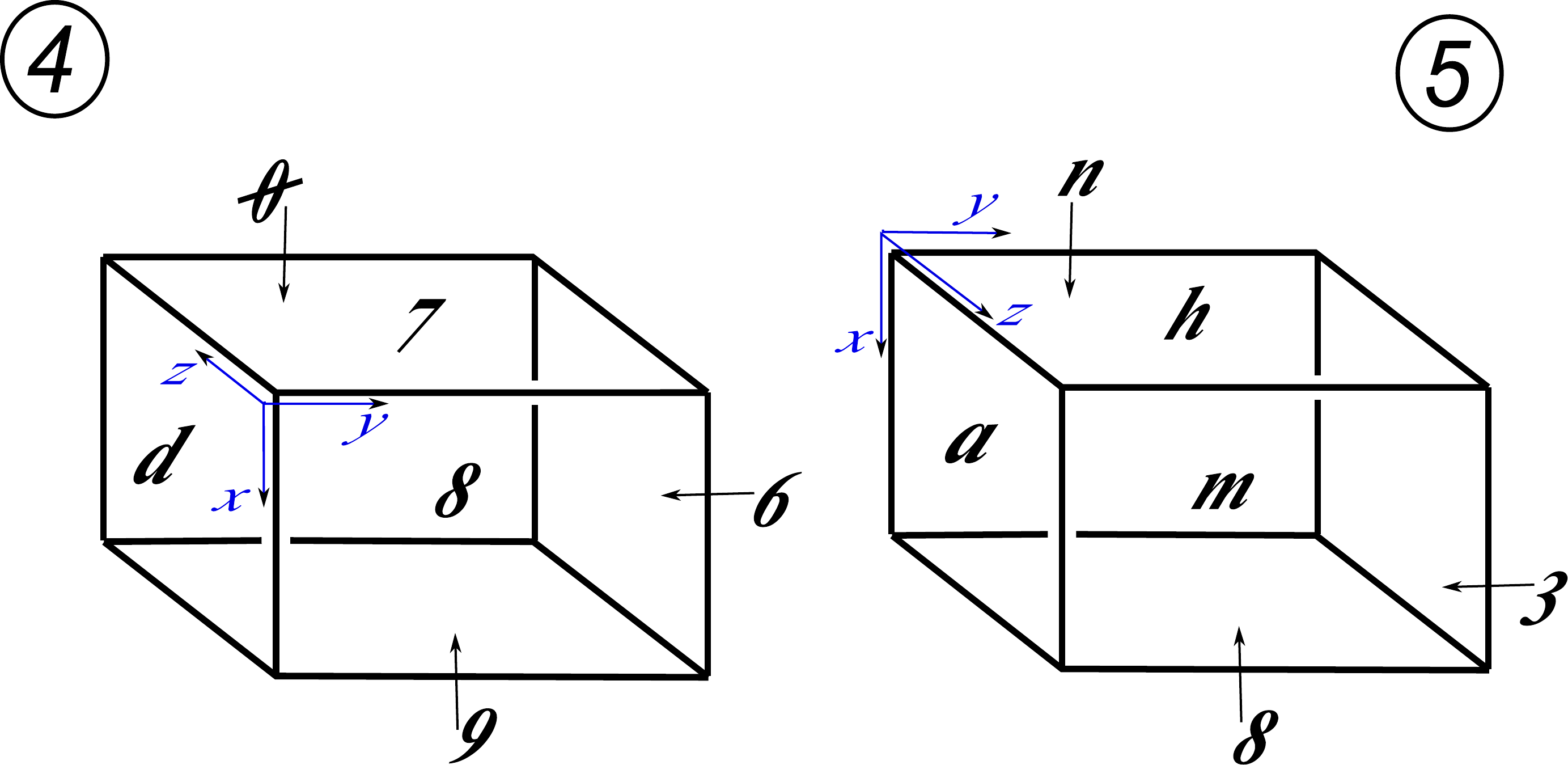}}
  \subfigure[]{\includegraphics[width = 6.2cm]{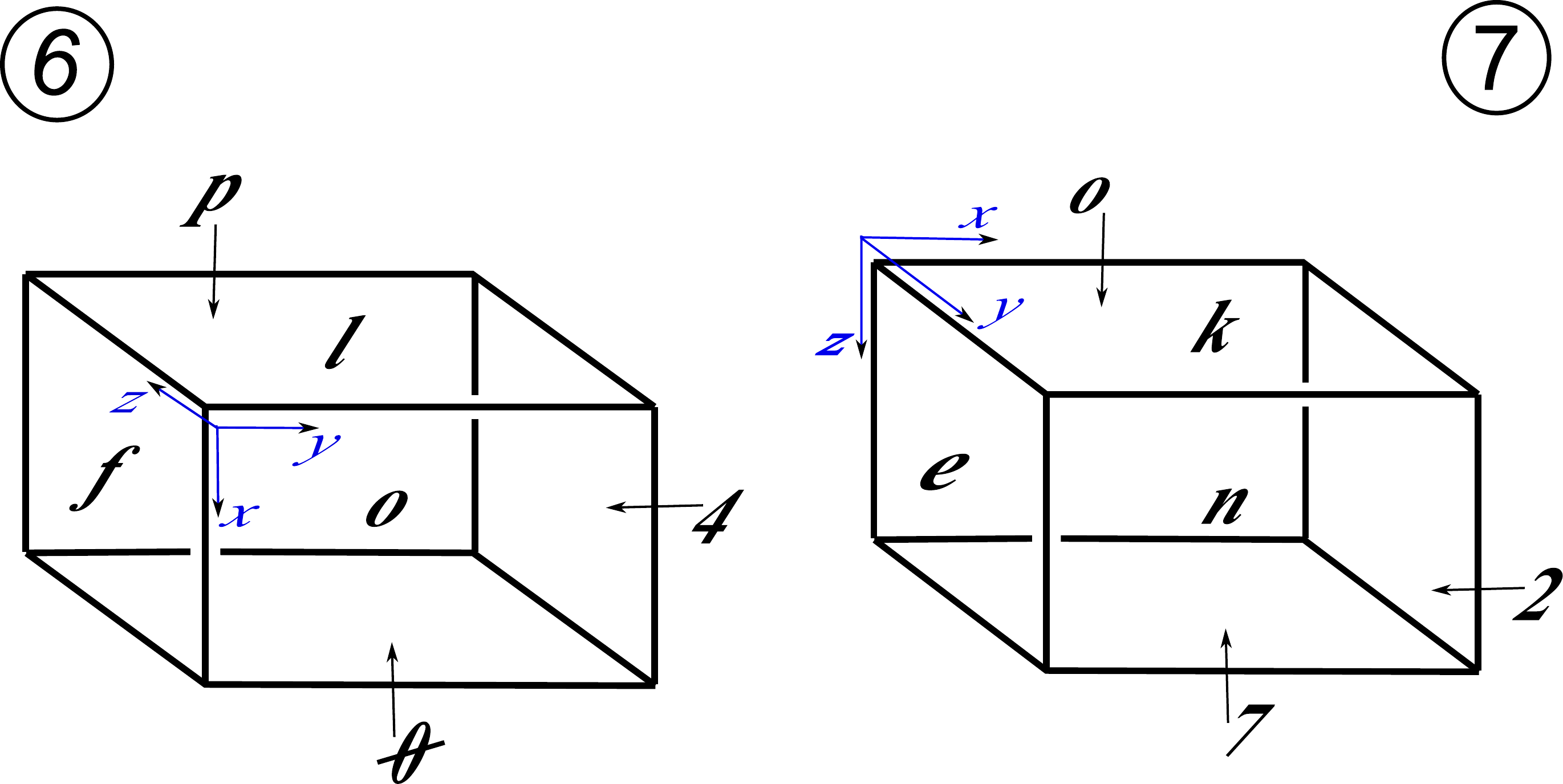}}
 \end{center}
 \caption{The cubulation giving rise to the manifold $\mathscr{K}$ from Proposition~\ref{2cusps:prop}}
 \label{cubulationK:fig}
\end{figure}

\begin{figure}
 \begin{center}
  \subfigure[]{\includegraphics[width = 6.2cm]{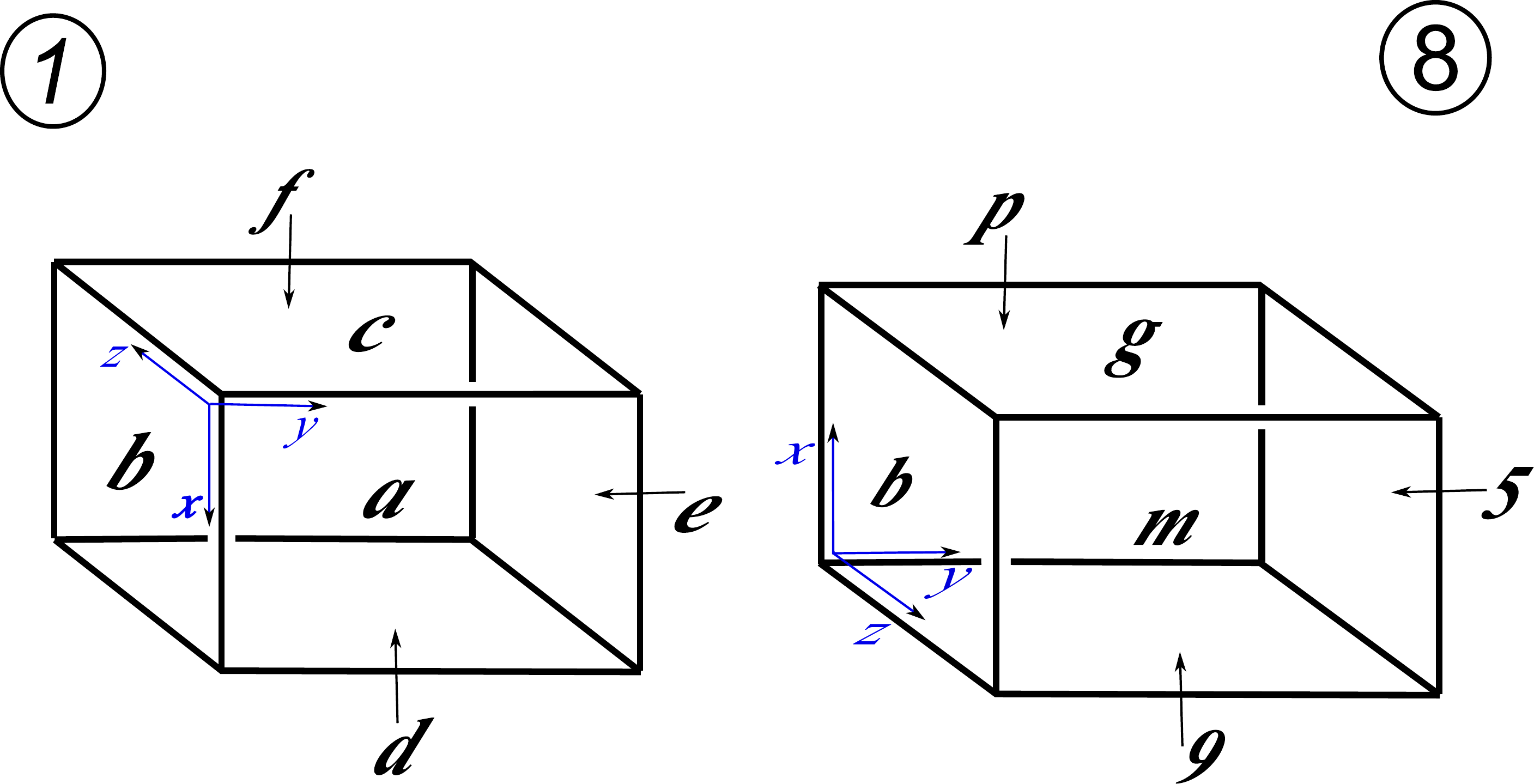}}
  \subfigure[]{\includegraphics[width = 6.2cm]{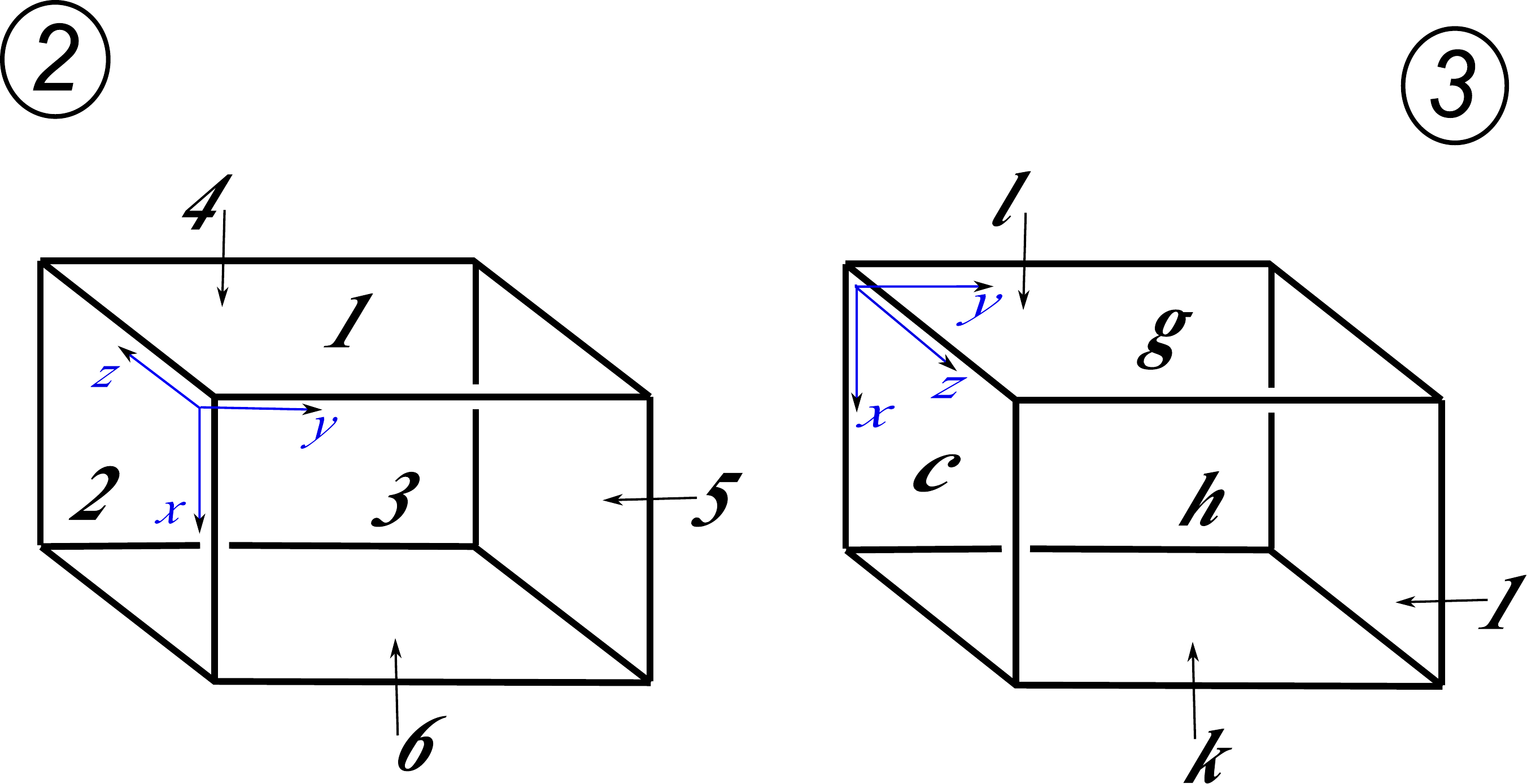}}
  \subfigure[]{\includegraphics[width = 6.2cm]{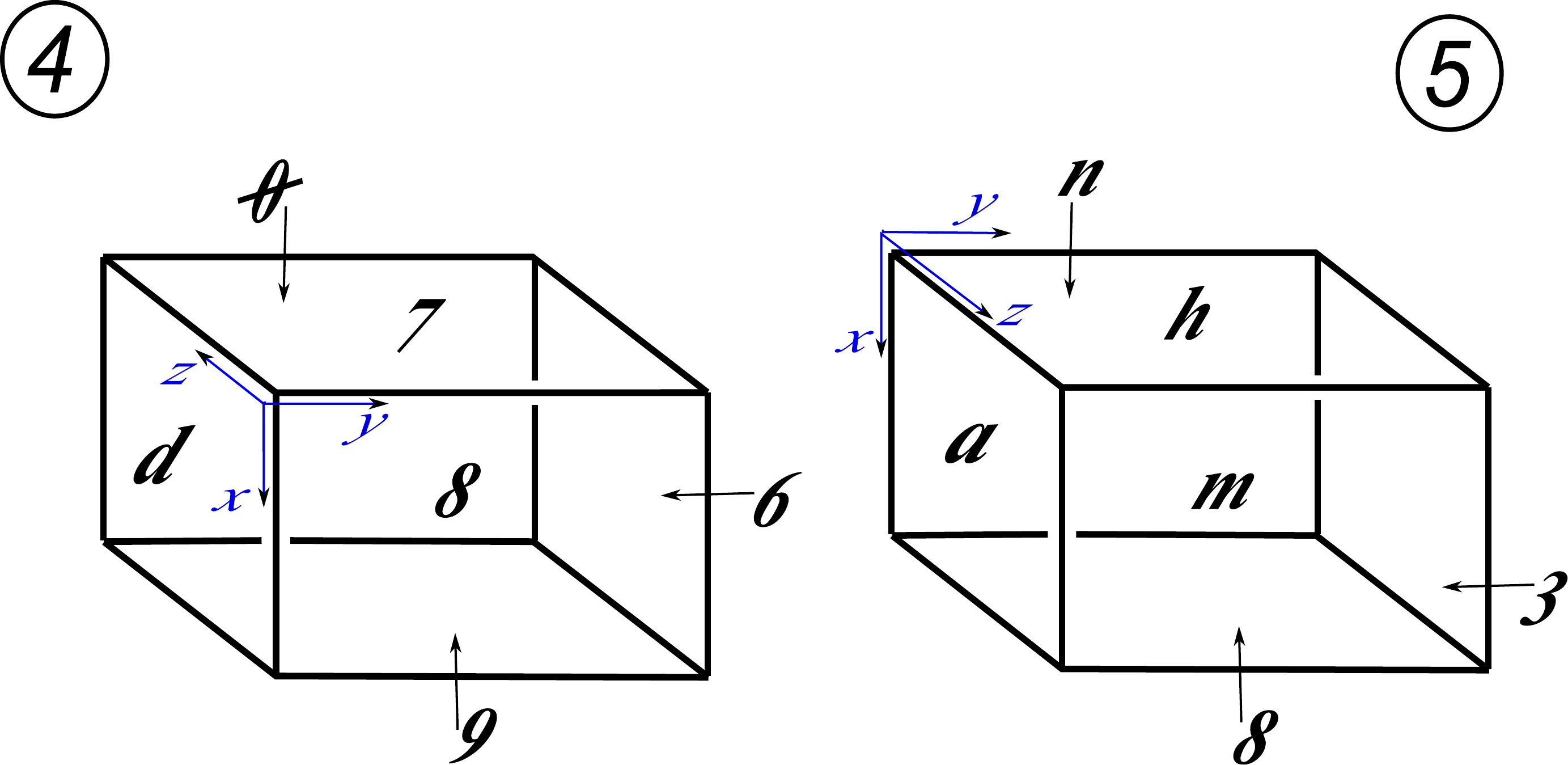}}
  \subfigure[]{\includegraphics[width = 6.2cm]{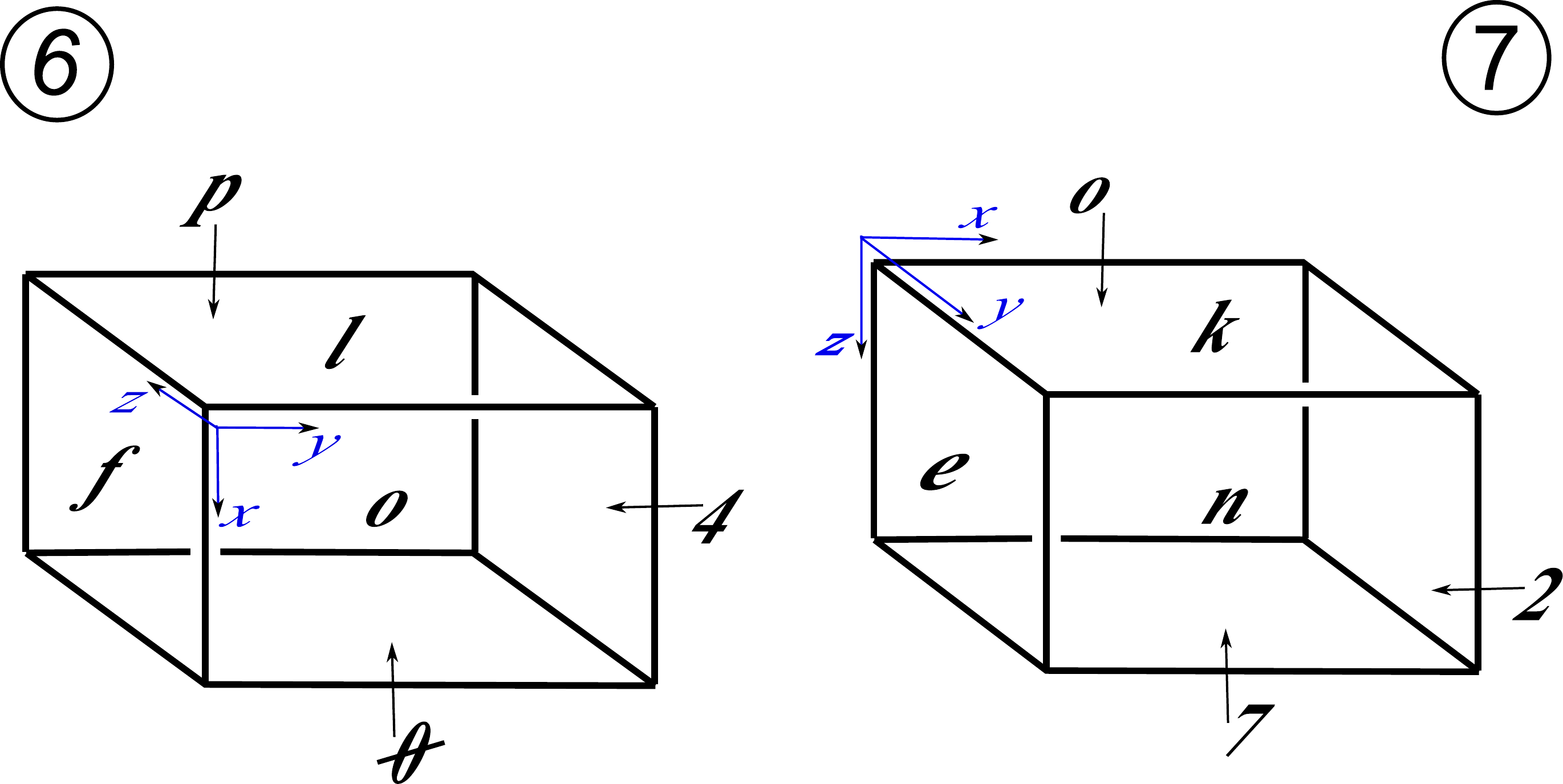}}
 \end{center}
 \caption{The cubulation giving rise to the manifold $\mathscr{L}$ from Proposition~\ref{2cusps:prop}}
 \label{cubulationL:fig}
\end{figure}

\begin{figure}
 \begin{center}
  \subfigure[]{\includegraphics[width = 6.2cm]{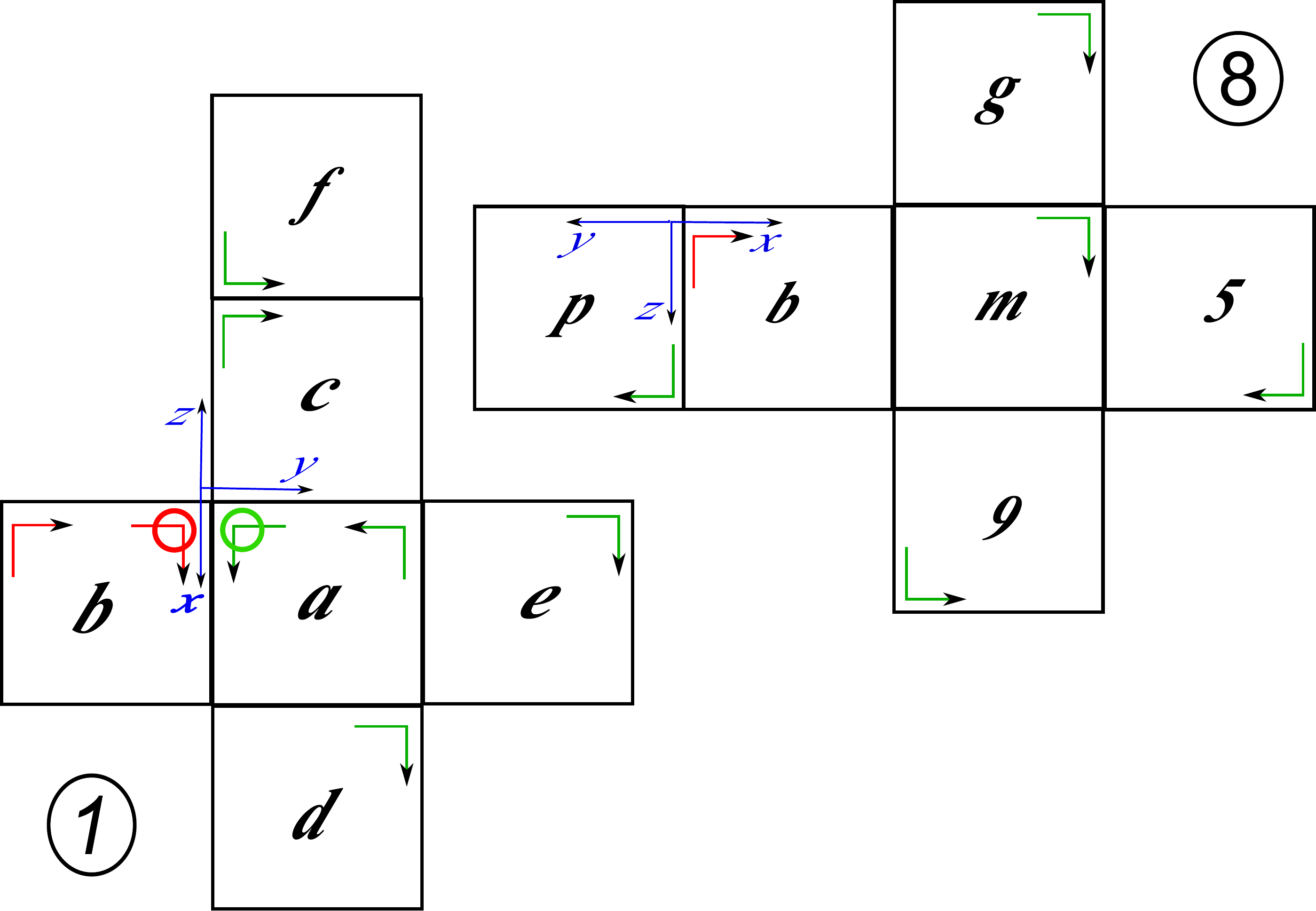}}
  \subfigure[]{\includegraphics[width = 6.2cm]{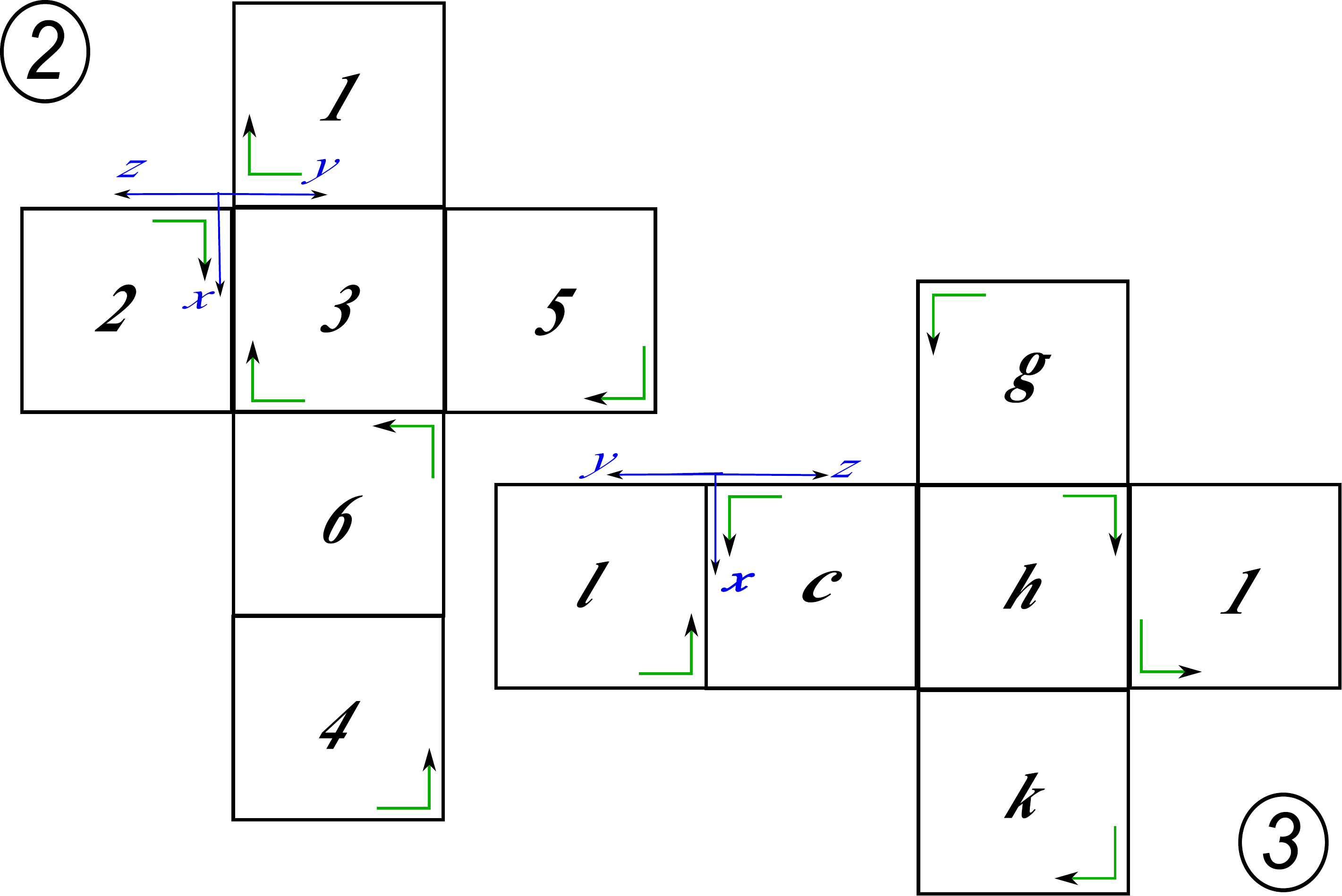}}
  \subfigure[]{\includegraphics[width = 6.2cm]{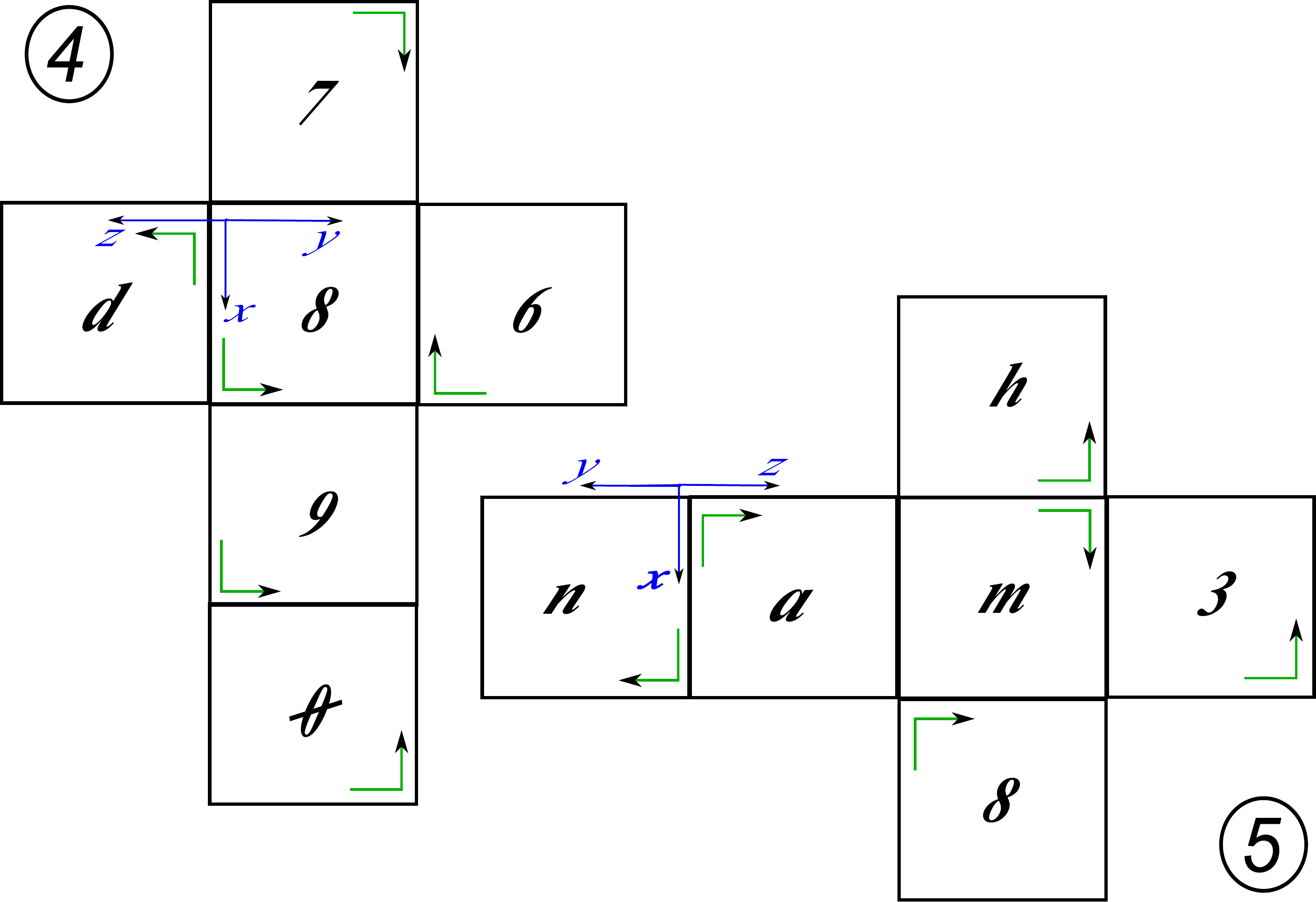}}
  \subfigure[]{\includegraphics[width = 6.2cm]{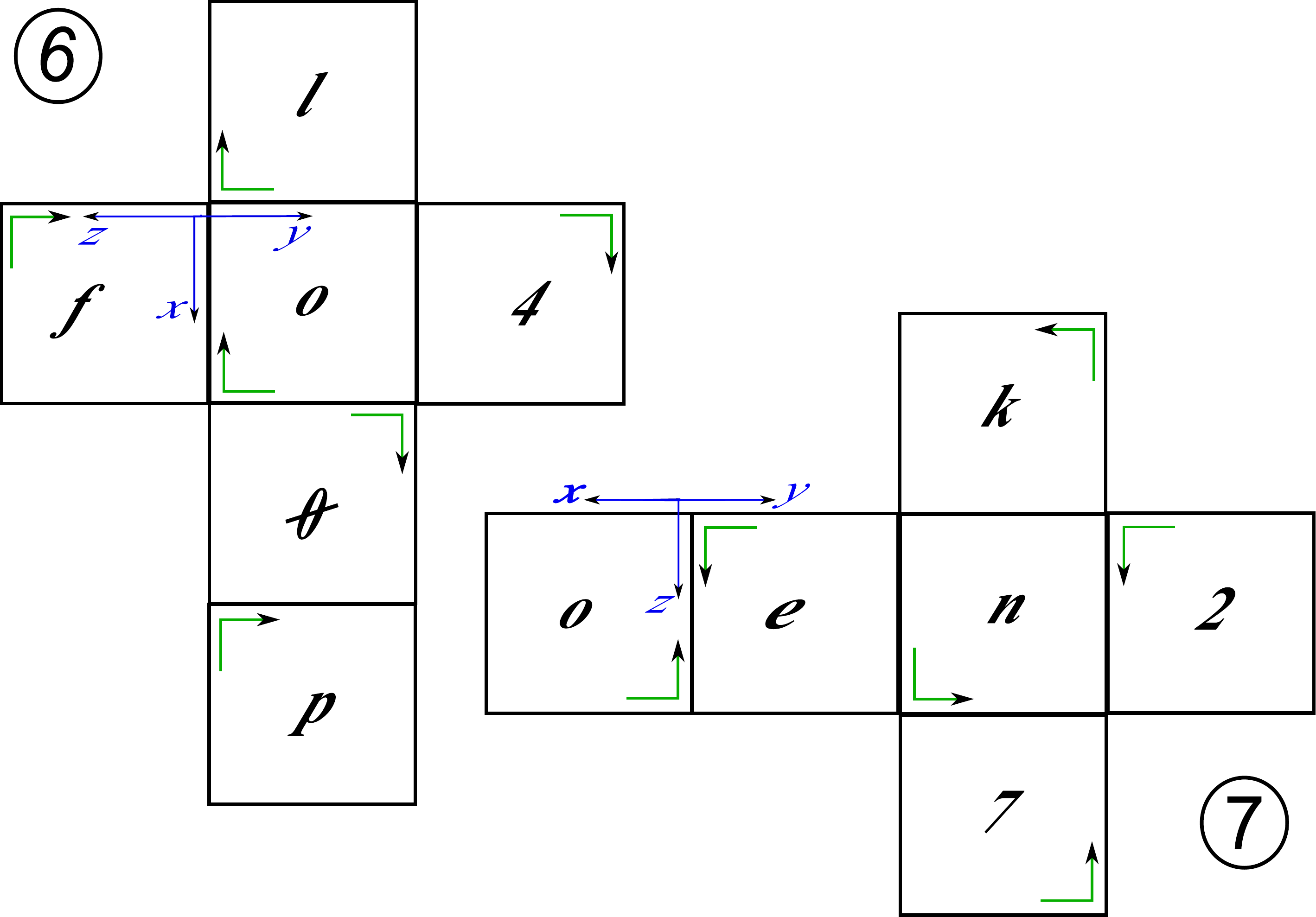}}
 \end{center}
 \caption{The cycles of $2$-faces, giving rise to the cusps of the manifold $\mathscr{K}$ from Proposition~\ref{2cusps:prop}. The frames that allow us to compute the monodromy are depicted}
 \label{cubulationK-monodromy:fig}
\end{figure}

\begin{figure}
 \begin{center}
  \subfigure[]{\includegraphics[width = 6.2cm]{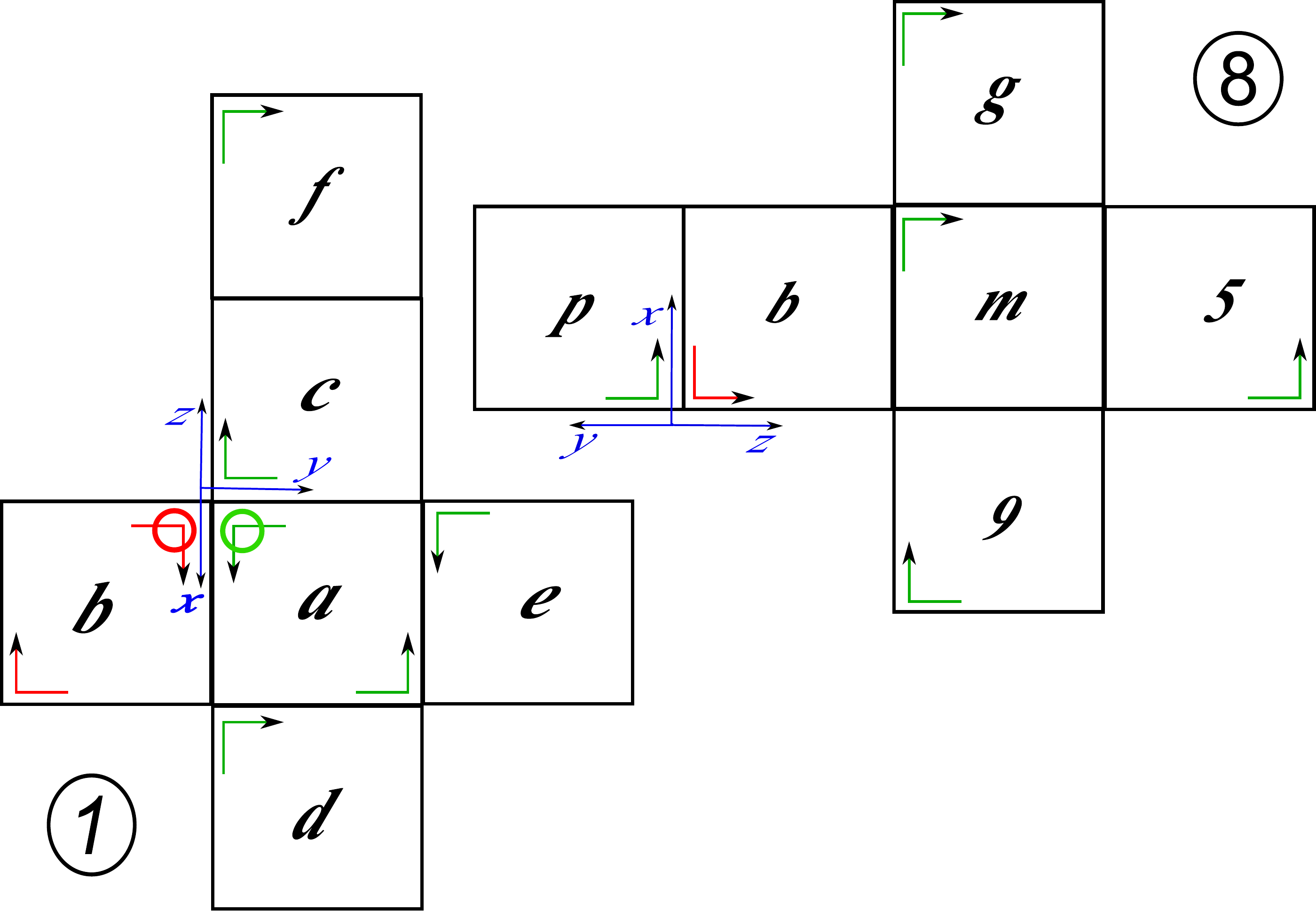}}
  \subfigure[]{\includegraphics[width = 6.2cm]{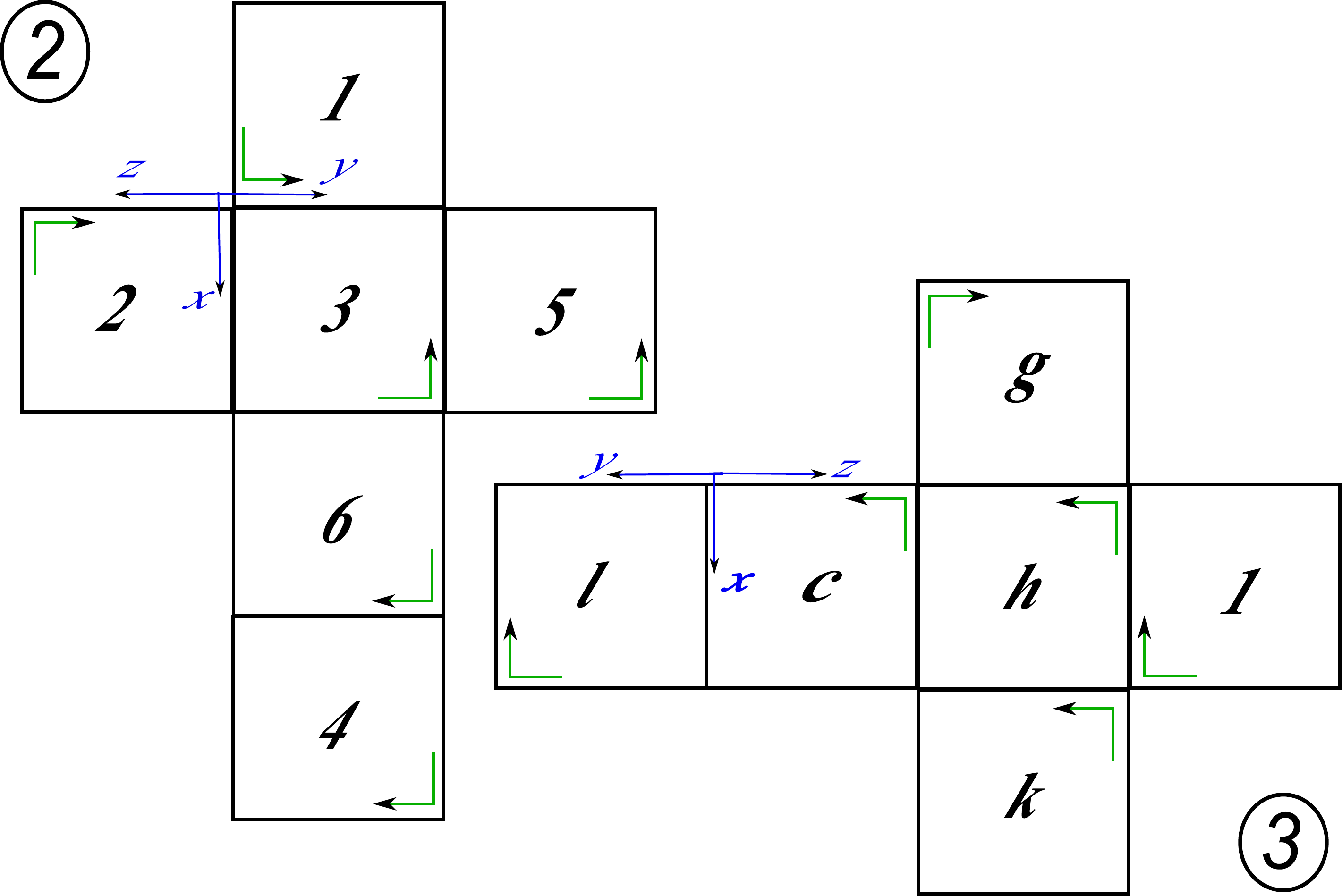}}
  \subfigure[]{\includegraphics[width = 6.2cm]{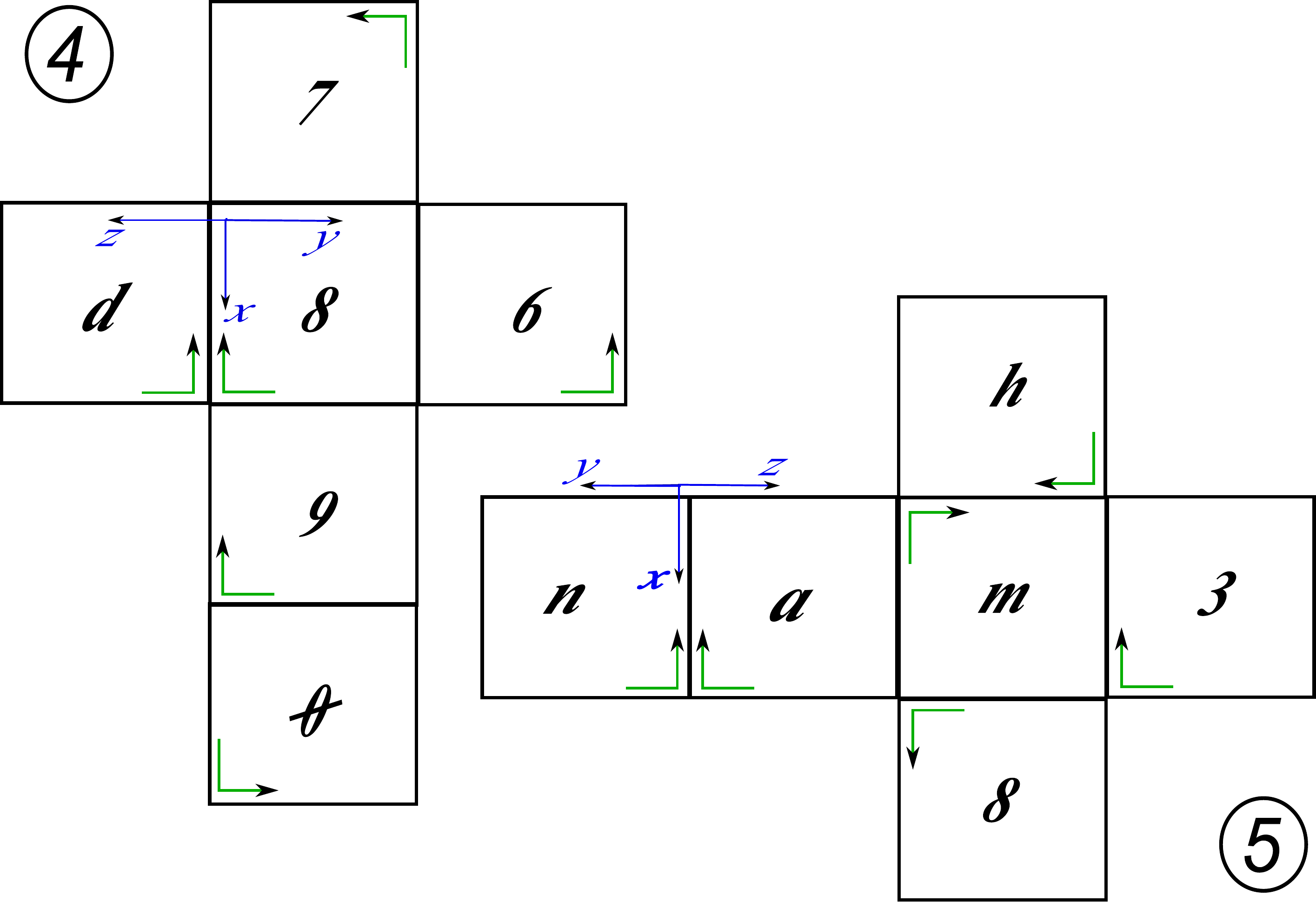}}
  \subfigure[]{\includegraphics[width = 6.2cm]{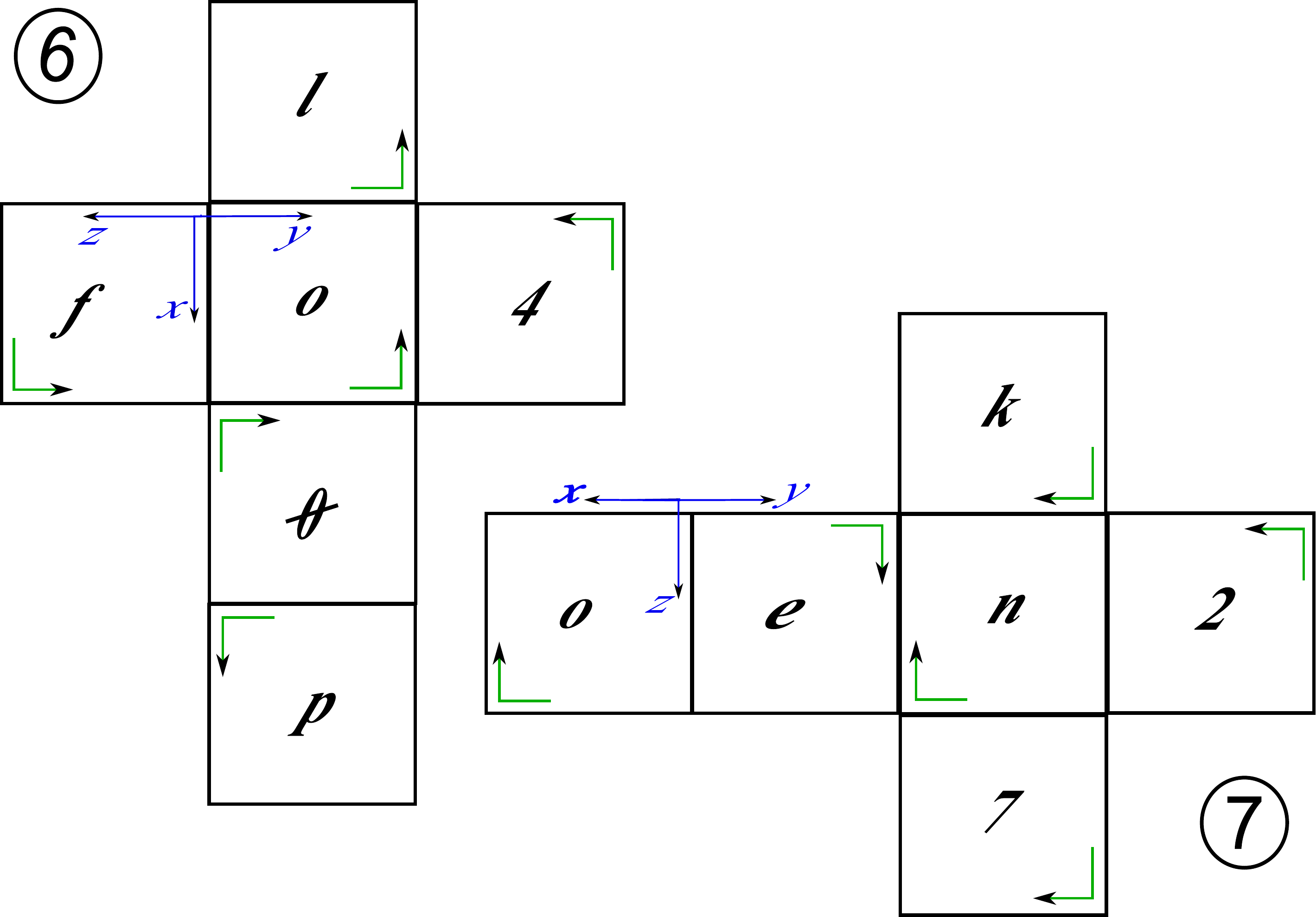}}
 \end{center}
 \caption{The cycles of $2$-faces, giving rise to the cusps of the manifold $\mathscr{L}$ from Proposition~\ref{2cusps:prop}. The frames that allow us to compute the monodromy are depicted.}
 \label{cubulationL-monodromy:fig}
\end{figure}

Using Ontaneda's terminology \cite{O}, we may say that $X\sqcup X$ bounds geometrically a hyperbolic manifold when $X$ is a torus bundle with monodromy $\left(\begin{smallmatrix} -1 & 0 \\ 0 & -1 \end{smallmatrix}\right)$ or  $\left(\begin{smallmatrix} 0 & 1 \\ -1 & 0 \end{smallmatrix}\right)$. We were not able to prove that $X$ bounds a hyperbolic manifold with our constructions; in fact, we formulate the following conjecture:
\begin{conj}
The cusp section of a hyperbolic $4$-manifold with a single cusp arising from a cubulation is a flat $3$-torus.
\end{conj}

Finally, it is easy to construct for every integer $n\geqslant 1$ a $n$-sheeted covering $\mathscr M_n$ of $\mathscr M$ with $n$ cusps, whose sections are $n$ flat tori. It suffices to take $n$ copies $H_0,\ldots, H_{n-1}$ of the hypercube $H$ shown in Fig.~\ref{cube:fig} and pair their facets as follows: 
$$(1_i, 2_i), \quad (3_i, 4_i), \quad (5_i, 6_i), \quad (7_i, 8_{i+1})$$
where $N_i$ indicates the facet number $N$ in $H_i$ and $i+1$ means addition modulo $n$. For every pair, the isometry is again the one described in Fig.~\ref{cubulationM:fig}. 
\begin{prop} The resulting hyperbolic manifold $\mathscr M_n$ has $n$ cusps. Its maximal cusp section consists of $n$ flat 3-tori, each obtained by identifying the opposite faces of a right-angled parallelepiped of size $2\times 2\times 24$.
\end{prop}
\begin{proof}
The square 2-faces are identified according to the following patterns, obtained by adding subscripts to \eqref{cubulationM:glueing1}-\eqref{cubulationM:glueing2}:
\begin{align}\label{cubulationM:glueing3}
\begin{array}{cccc}
\mathbf{a}_i\rightarrow \mathbf{1}_i& \mathbf{f}_i\rightarrow \mathbf{6}_i& \mathbf{1}_i\rightarrow \mathbf{8}_i& \mathbf{c}_i\rightarrow \mathbf{\emptyset}_i\\
\mathbf{b}_i\rightarrow \mathbf{2}_i& \mathbf{e}_i\rightarrow \mathbf{5}_i& \mathbf{g}_i\rightarrow \mathbf{9}_i& \mathbf{k}_i\rightarrow \mathbf{7}_i\\
\mathbf{c}_i\rightarrow \mathbf{3}_i& \mathbf{d}_i\rightarrow \mathbf{4}_i& \mathbf{h}_i\rightarrow \mathbf{d}_i& \mathbf{l}_i\rightarrow \mathbf{6}_i
\end{array}
\end{align}
\begin{align}\label{cubulationM:glueing4}
\begin{array}{cccc}
\mathbf{m}_i\rightarrow \mathbf{\emptyset}_i& \mathbf{n}_i\rightarrow \mathbf{l}_i& \mathbf{e}_i\rightarrow \mathbf{9}_{i+ 1}& \mathbf{2}_i\rightarrow \mathbf{g}_{i+ 1}\\
\mathbf{3}_i\rightarrow \mathbf{f}_i& \mathbf{a}_i\rightarrow \mathbf{4}_i& \mathbf{n}_i\rightarrow \mathbf{5}_{i+ 1}& \mathbf{o}_i\rightarrow \mathbf{b}_{i+ 1}\\
\mathbf{h}_i\rightarrow \mathbf{o}_i& \mathbf{8}_i\rightarrow \mathbf{p}_i& \mathbf{k}_i\rightarrow \mathbf{m}_{i+ 1}& \mathbf{7}_i\rightarrow \mathbf{p}_{i+ 1}.
\end{array}
\end{align}

The cubulation has exactly $n$ cycles of square $2$-faces, each of the form
\begin{align*}
\mathbf{a}_i\mathbf{1}_i\mathbf{8}_i\mathbf{p}_i\mathbf{7}_{i- 1}\mathbf{k}_{i- 1}\mathbf{m}_i\mathbf{\emptyset}_i
\mathbf{c}_i\mathbf{3}_i\mathbf{f}_i\mathbf{6}_i\mathbf{l}_i\mathbf{n}_i\mathbf{5}_{i+ 1}\mathbf{e}_{i+ 1}\mathbf{9}_{i+ 2}\mathbf{g}_{i+ 2}
\mathbf{2}_{i+ 1}\mathbf{b}_{i+ 1}\mathbf{o}_i\mathbf{h}_i\mathbf{d}_i\mathbf{4}_i\mathbf{a}_i.
\end{align*}
The cycle is analogous to (\ref{cubulationM:cycle}) and thus has trivial monodromy.
\end{proof}

\begin{cor}
For every integer $n$ there is a hyperbolic orientable four-manifold with $n$ cusps, all whose sections are $3$-tori.
\end{cor}

\section{Dehn filling} \label{Dehn:section}
In the previous sections we have developed an algorithm that transforms an orientable cubulation $C$ into a hyperbolic four-manifold $M$. Here we assume that the monodromy of every cycle of $2$-dimensional faces is trivial, so every cusp section is isometric to the $3$-torus obtained by identifying the opposite faces of the right-angled parallelepiped $[-1,1]\times [-1,1] \times [0,h]$ (see for instance Examples \ref{1:example} and \ref{2:example}).

The manifold $M$ is the interior of a compact manifold $\overline M$ with boundary that consists of $3$-tori. Let us consider one boundary component $X\subset \partial \overline M$. A \emph{Dehn filling} on $X$ is the topological operation that consists of attaching a copy of $D^2 \times T^2$ along $X$. The resulting smooth manifold is determined only by the homotopy class of the closed curve $\partial D^2 \times \{{\rm pt}\}$ in $\pi_1(X)=H_1(X,\matZ)$. If we fix a basis for the homology, we identify $H_1(X,\matZ)$ with $\matZ^3$ and the Dehn filling is determined by a triple $(p,q,r)$ of co-prime integers. A natural basis here is given by the three sides of the parallelepiped.

It is then possible to encode the Dehn fillings of $M$ by assigning a triple $(p,q,r)$ at each cycle of $2$-dimensional faces. This determines a curve in the $3$-torus $[-1,1]\times [-1,1] \times [0,h]/_\sim$ having length $\ell = \sqrt {(2p)^2 + (2q)^2 + (hr)^2}$.

By Thurston-Gromov's $2\pi$-theorem, which holds in all dimensions (see for instance \cite{A}), whenever $\ell \geqslant 2\pi$ the Dehn filled manifold admits a non-positively curved metric and is hence in particular aspherical by Cartan-Hadamard theorem. We can therefore construct plenty of non-positively curved four-manifolds from a simple combinatorial datum: a cubulation where each cycle of $2$-faces has trivial monodromy and is assigned a triple of co-prime numbers $(p,q,r)$ such that $p^2+q^2+(hr/2)^2 > \pi^2$.

Recently, M.~Anderson has extended Thurston's Dehn filling theorem by showing that if $\ell$ is big enough then the filled manifold admits an Einstein metric \cite{A}, see also \cite{B}. We can therefore also construct many such manifolds from a simple combinatorial datum. 

Let $\sigma$ and $\|\cdot \|$ denote the signature and Gromov norm, respectively, and let $n$ denote the number of hypercubes in the original cubulation. Let $v_4$ be the volume of the ideal regular hyperbolic $4$-simplex.

\begin{prop} The Dehn filled manifold $M^{\rm fill}$ has 
\begin{align*}
\chi(M^{\rm fill}) & = \chi(M) = 4n, \\
\sigma(M^{\rm fill}) & = \sigma(M) =0, \\
\|M^{\rm fill}\| & \leqslant \|M\| = \frac{\Vol(M)}{v_4} = \frac{16n}{3v_4}\pi^2.
\end{align*}
\end{prop}
\begin{proof}
Both the Euler characteristic and the signature are additive on the pieces when we glue $T^2\times D^2$ to $\overline M$ by Novikov's additivity theorem \cite{Nov}. We have $\sigma(T^2\times D^2) = \chi(T^2\times D^2)=0$, and hence $\chi (M^{\rm fill}) = \chi(M)= 4n$ by Proposition \ref{volume:prop} and $\sigma(M^{\rm fill}) = \sigma (M) = -\eta (\partial \overline M)$ by \cite{LR}. The boundary of $\overline M$ consists of $3$-tori, that are mirrorable and hence their $\eta$-invariant vanishes.

The Gromov norm of $M^{\rm fill}$ is not bigger than that of $M$ by \cite{FM} and we have $\|M\| = \frac{\Vol(M)}{v_4} =  \frac{16n}{3v_4}\pi^2$.
\end{proof}

Via cubulations we may also construct certain hyperbolic four-manifolds $M_1$, $\dots$, $M_k$, representing the interiors of compact ones, say $\overline{M_1}, \dots, \overline{M_k}$, with toric boundaries, and then pair the boundary tori along some diffeomorphisms, obtaining various graph manifolds \cite{FLS}.


\begin{thebibliography}{99}

\bibitem{A} \textsc{M.~Anderson}, \emph{Dehn filling and Einstein metrics in higher dimensions}, J.~Diff.~Geom. \textbf{73} (2006), 219--261.

\bibitem{B} \textsc{R.H.~Bamler}, \emph{Construction of Einstein metrics by 
generalized Dehn filling}, J.~Eur. Math.~Soc.~\textbf{14}, 887--909.

\bibitem{Bollobas} \textsc{B. Bollob\`{a}s}, \emph{The asymptotic number of unlabelled regular graphs}, J.~London Math.~Soc. \textbf{26} (1982), 201--206.

\bibitem{BGLM} \textsc{M.~Burger -- T.~Gelander -- A.~Lubotzky -- S.~Mozes}, \emph{Counting hyperbolic manifolds}, Geom. \& Funct. Anal., 
\textbf{12}, (2002), 1161--1173. 

\bibitem{CFMP} \textsc{F.~Costantino -- R.~Frigerio -- B.~Martelli -- C.~Petronio}, \emph{Triangulations of $3$-manifolds, hyperbolic relative handlebodies, and Dehn filling},
Comm. Math. Helv. \textbf{82} (2007), 903--934. 

\bibitem{SnapPy} \textsc{M.~Culler -- N.~Dunfield -- J.~Weeks}, \emph{SnapPy}, a computer 
program for studying the geometry and topology of $3$-manifolds, 
{\tt http://snappy.computop.org/}

\bibitem{G} \textsc{M.~Gromov}, \emph{Volume and bounded cohomology}, Publ. IHES, \textbf{56} (1982), 5--99.

\bibitem{E} \textsc{D.B.A.~Epstein -- R.C.~Penner}, \emph{Euclidean decompositions of non-compact hyperbolic manifolds}, J. Diff. Geom. \textbf{27} (1988), 67--80.

\bibitem{FLS} \textsc{R.~Frigerio -- J.~Lafont -- A.~Sisto}, \emph{Rigidity of high dimensional graph manifolds}, {\tt arXiv:1107.2019}

\bibitem{FM} \textsc{K.~Fujiwara -- J.F.~Manning}, \emph{Simplicial volume and fillings of hyperbolic manifolds}, Algebraic \& Geometric Topology \textbf{11} (2011), 2237--2264.

\bibitem{K} \textsc{A.~Kolpakov}, \emph{On the optimality of the ideal right-angled $24$-cell}, Alg. Geom. Topol., \textbf{12} (2012), 1941-1960.

\bibitem{LR} \textsc{D.D.~Long -- A.W.~Reid}, \emph{On the geometric boundaries of hyperbolic 
4-manifolds}, Geom. Topol. \textbf{4} (2000) 171--178.

\bibitem{LRall} \bysame, \emph{All flat manifolds are cusps of hyperbolic orbifolds}, Alg. Geom. Topol. \textbf{2} (2002), 285--296.

\bibitem{N} \textsc{B.E.~Nimershiem}, \emph{All flat three-manifolds appear as cusps of hyperbolic four-manifolds}, Topology and Its Appl. \textbf{90} (1998), 109--133. 

\bibitem{Nov} \textsc{S.P.~Novikov}, \emph{Pontrjagin Classes, the Fundamental Group and some Problems of Stable Algebra}, Essays on Topology and Related Topics (1970) 147-155; available from {\tt http://www.maths.ed.ac.uk/~aar/papers/novstable.pdf}

\bibitem{O} \textsc{P.~Ontaneda}, \emph{Pinched Smooth Hyperbolization}, {\tt arXiv:1110.6374}

\bibitem{RT} \textsc{G.J.~Ratcliffe -- S.T.~Tschantz}, \emph{The volume spectrum of hyperbolic 4-manifolds}, Experimental Math. \textbf{9} (2000), 101--125.

\bibitem{Stover} \textsc{M.~Stover}, \emph{On the number of ends of rank one locally symmetric spaces}, Geom. Topol. \textbf{17} (2013), 905--924.

\bibitem{Th} \textsc{W.P.~Thurston}, ``Geometry and Topology of $3$-Manifolds,'' mimeographed notes, Princeton University, 1979; available from {\tt www.msri.org/publications/books/gt3m/}
\end{thebibliography}
\end{document}